\documentclass{amsart}
\usepackage{amsfonts,amssymb,enumitem,latexsym,mathdots,mathtools}

\theoremstyle{plain}

\newtheorem{theorem}{Theorem}
\newtheorem{theorem*}{Theorem}
\newtheorem{lemma}{Lemma}

\newtheorem{proposition}{Proposition}

\theoremstyle{definition}

\newtheorem{remark}{Remark}

\newtheorem{definition}{Definition}

\allowdisplaybreaks[1]

\DeclareMathOperator{\rank}{rank}
\DeclareMathOperator{\Null}{null}

\begin{document}

\title%
[Canonical forms for boundary conditions]%
{Canonical forms for boundary conditions of self-adjoint odd-order differential operators}

\author{Yorick Hardy}  
\author{Bertin Zinsou}
\email[Y.~Hardy]{yorick.hardy@wits.ac.za}
\email[B.~Zinsou]{bertin.zinsou@wits.ac.za}

\address{
 School of Mathematics,
 University of the Witwatersrand,
 Private Bag 3, Wits 2050, South Africa
}
\address{
 National Institute for Theoretical and Computational Sciences (NITheCS),
 South Africa
}

\begin{abstract}
 It is useful to have canonical forms of boundary
 conditions in the study of the eigenvalues of boundary
 value problems and associated numerical applications.
 In [J. Appl. Anal. Comput., 2024, 14(4), {1854--1868}],
 a canonical form is given for self-adjoint
 differential operators of even order, with eigenvalue
 parameter dependent boundary conditions. In this
 article, we derive canonical forms for the remaining
 case, namely: for self-adjoint $(2n+1)$-th order
 differential operators with eigenvalue parameter
 dependent boundary conditions.
\end{abstract}



\maketitle

\section{Introduction}
Canonical forms of boundary conditions are important in the study of the eigenvalues of boundary conditions \cite[Section 1.3]{zettl21}, their numerical computations \cite{bail}, characterizations of self-adjoint extensions of systems \cite{zemanek16} and the study of transmission conditions \cite{zinsou20}. These canonical forms are of importance in the study of dependence of eigenvalues with respect to the parameters in the differential equation and the boundary conditions \cite{li17,li23,lin23,zheng23}. In particular, this dependence arises in $n$-th order boundary value transmission problems \cite{li17,lin23,hinton78}. Previous work in \cite{hao12,bao,niu20} provided canonical forms for boundary conditions of operators for order up to 4. In \cite{hardy24}, a general canonical form was given for boundary conditions of operators of even order; but a (general) odd-order canonical form was not studied. This paper provides a canonical form for boundary conditions for the remaining odd-order cases which have not been studied in previous work.

In this paper, we extend the study conducted in \cite{hardy24} to $(2n+1)$-th order differential operators using similar methods. We start our investigation with fifth order differential operators with self-adjoint boundary conditions that we extend to $(2n+1)$-th differential operators with self-adjoint boundary conditions. This article also provides a new proof for \cite[Theorem 3]{zhang23}.

In Section \ref{fifth}, we introduce the self-adjoint fifth order differential operators with eigenvalue dependent boundary conditions under consideration. In Section \ref{types}, we present a brief discussion on the types of boundary conditions for the self-adjoint fifth order differential operators. Next we classify the different types of canonical forms for self-adjoint fifth order differential operators in Section \ref{cano} using a new canonical form that we extend to all self-adjoint odd-order differential operators in Section \ref{cano2n+1}.

%%%%%%%%%%%%%%%%%%%%%%%%%%%%%%%%%%%%%%%%%%%%%%%%%%%%%%%%%%%%%%%%%%%%%%%%%%%%%%%%%%%%%%%%%%%%%%%%%%%%%%%%%%%%%%%%%%%%%%%%%%%%%%%%%%%%%%

%%%%%%%%%%%%%%%%%%%%%%%%%%%%%%%%%%%%%%%%%%%%%%%%%%%%%%%%%%%%%%%%%%%%%%%%%%%%%%%%%%%%%%%%%%%%%%%%%%%%%%%%%%%%%%%%%%%%%%%%%%%%%%%%%

\section{Self-adjoint fifth order boundary value problems} \label{fifth}

We consider on the interval $J=(a,b)$, $-\infty\le a<b\le \infty$, the fifth order differential equation 
with formally self-adjoint differential expression (with smooth coefficients) \cite{hinton78,lin23}
%\begin{multline}\label{fiftheq1}My=i(q_2y)'-i(q_1y')''+i(q_0y'')'''\\ +i(q_2y')-i(q_1y'')'+i(q_0y''')''\\ +i(p_2y)-i(p_1y')'+i(p_0y'')''=\lambda wy,\end{multline}
\begin{equation}\label{fiftheq1}My=\sum_{k=0}^2 (-1)^k\left(i\left[\left(q_{2-k}y^{(k)}\right)^{(k+1)}+\left(q_{2-k}y^{(k+1)}\right)^{(k)}\right]+\left(p_{2-k}y^{(k)}\right)^{(k)}\right)=\lambda wy,\end{equation}
where $q_0^{-1}$ exists on $J$, $p_j,q_j\in C^j(J)$ are sufficiently smooth real-valued functions on $J$ and
$w\in L(J,\mathbb R)$ is a real-valued Lebesgue integrable function on $J$, $w>0$ a.e. on $J$.
Let
\begin{equation*}
 \theta:=\dfrac1{\sqrt2}+\dfrac{i}{\sqrt2}.
\end{equation*}
If the coefficients in \eqref{fiftheq1} are not smooth, we introduce the quasi-derivatives
of $y$ given by
\begin{align*}
 &y^{[1]} = y',\\
 &y^{[2]} = -\theta q_0(y^{[1]})', \\
 &y^{[3]} = -\theta q_0(y^{[2]})'+i\theta{p_0}{q_0}^{-1}y^{[2]}-iq_1y^{[1]}, \\
 &y^{[4]} = -(y^{[3]})'-\theta{q_1}{q_0}^{-1}y^{[2]}+p_1y^{[1]}-iq_2y, \\
 &y^{[5]} = -(y^{[4]})'-iq_2y'+p_2y.
\end{align*}
and \eqref{fiftheq1} is replaced by the equation $y^{[5]}=\lambda wy$
where $q_0^{-1}, p_0, p_1, p_2, q_1, q_2, w\in L(J,\mathbb R)$, $p_2>0$, $w>0$ a.e. on $J$
\cite{hinton78}.
Let $Y=\left(y, y', y'', y^{[3]}, y^{[4]}\right)^{\top}$. We now consider the fifth order boundary value problem defined by \eqref{fiftheq1} and the boundary conditions
\begin{align}\label{fiftheq2}AY(a)+BY(b)=0, \quad A, B\in M_5(\mathbb C).\end{align}
We note that \eqref{fiftheq2} is invariant under multiplication by invertible matrices.
We will write
\begin{equation*}
 (A:B) \sim (A':B')\qquad\Longleftrightarrow\qquad \text{there exists $J$ such that~} J^{-1}(A:B) = (A':B').
\end{equation*}
Furthermore, we will write $A\sim A'$ and $B\sim B'$ (jointly) to mean that $(A:B)\sim(A':B')$. This
notation will be clear in the context.

For the boundary conditions \eqref{fiftheq2} with the assumptions made so far, \cite[Theorem 2.4]{MolZet2} leads to  
\begin{proposition}\label{fifthprop1}Let $C_5$ be the symplectic matrix of order 5 defined by \begin{align}\label{fiftheq4} C_5=\left((-1)^r\delta_{r,6-s}\right)_{r,s=1}^5,\end{align} where  $\delta$ is the Kronecker delta. Then   problems \eqref{fiftheq1}--\eqref{fiftheq2} are self-adjoint if and only if 
\begin{align}\label{fiftheq3}\rank(A:B)=5 \quad \textrm{and}\quad AC_5A^*=BC_5B^*. \end{align}\end{proposition}

%%%%%%%%%%%%%%%%%%%%%%%%%%%%%%%%%%%%%%%%%%%%%%%%%%%%%%%%%%%%%%%%%%%%%%%%%%%%%%%%%%%%%%%%%%%%%%%%%%%%%%%%%%%%%%%%%%%%%%%%%%%%%%%%%%%%%%

%%%%%%%%%%%%%%%%%%%%%%%%%%%%%%%%%%%%%%%%%%%%%%%%%%%%%%%%%%%%%%%%%%%%%%%%%%%%%%%%%%%%%%%%%%%%%%%%%%%%%%%%%%%%%%%%%%%%%%%%%%%%%%%%%

\section{Types of boundary conditions of fifth order differential operators}\label{types}

The following theorem (which follows from \cite[Theorem 3]{zhang23}) gives conditions satisfied by the matrices $A$, $B$ for the problems \eqref{fiftheq1}--\eqref{fiftheq2} to be self-adjoint.
\begin{theorem}\label{typesthm1}
 Assume that the matrices $A, B\in M_5(\mathbb C)$ satisfy \eqref{fiftheq3}. Then

{\rm (i)} $3\le \rank A\le 5$, $3\le \rank B \le 5$;

{\rm (ii)} let $0\le r \le 2$;  if $\rank A=3+r$, then $\rank B =3+r$.
\end{theorem}

We provide a simple, albeit technical, alternative proof of this theorem in Section \ref{sec:5thorder}.
Note that the boundary conditions \eqref{fiftheq2} are invariant under left multiplication by a non singular matrix $G\in M_5(\mathbb C)$ and if $ AC_5A^*=BC_5B^*$, then 
\begin{align*} (GA)C_5(GA)^*=(GB)C_5(GB)^*.\end{align*}
Therefore, the boundary condition form \eqref{fiftheq3} is invariant under elementary matrix row transformations of $(A:B)$.

Next, we define the different types of boundary conditions based on Theorem \ref{typesthm1}.

\begin{definition}\label{typesdef1}%
 Let the hypotheses and notation of Theorem \ref{typesthm1} hold.
 Then the boundary conditions \eqref{fiftheq2}, \eqref{fiftheq3} are 

\noindent
(1) mixed if $r=0$ or $r=1$,\\
(2) coupled if $r=2$.
 \end{definition}
\begin{remark}\label{typermk2} Note that the boundary conditions \eqref{fiftheq2} are coupled if each of the five boundary conditions involves both endpoints, while they are mixed if there is at least one separated and one coupled boundary condition. By \cite[Theorem 2]{zhang23}, fully separated boundary conditions are not possible in the 5th order case (we will also prove this independently in Theorems \ref{thm:2n+1th-odd} and \ref{thm:2n+1th-even} below).  \end{remark}

%%%%%%%%%%%%%%%%%%%%%%%%%%%%%%%%%%%%%%%%%%%%%%%%%%%%%%%%%%%%%%%%%%%%%%%%%%%%%%%%%%%%%%%%%%%%%%%%%%%%%%%%%%%%%%%%%%%%%%%%%%%%%%%%%%%%%%%
%
%%%%%%%%%%%%%%%%%%%%%%%%%%%%%%%%%%%%%%%%%%%%%%%%%%%%%%%%%%%%%%%%%%%%%%%%%%%%%%%%%%%%%%%%%%%%%%%%%%%%%%%%%%%%%%%%%%%%%%%%%%%%%%%%%%

%%%%%%%%%%%%%%%%%%%%%%%%%%%%%%%%%%%%%%%%%%%%%%%%%%%%%%%%%%%%%%%%%%%%%%%%%%%%%%%%%%%%%%%%%%%%%%%%%%%%%%%%%%%%%%%%%%%%%%%%%%%%%%%%%%%%%%

%%%%%%%%%%%%%%%%%%%%%%%%%%%%%%%%%%%%%%%%%%%%%%%%%%%%%%%%%%%%%%%%%%%%%%%%%%%%%%%%%%%%%%%%%%%%%%%%%%%%%%%%%%%%%%%%%%%%%%%%%%%%%%%%%

\section{Canonical forms for fifth order differential operators}
\label{cano}
\label{sec:5thorder}

Here we follow the same approach as presented in \cite{hardy24}, adjusted
for the odd-order case.
\begin{lemma}
 \label{lem:fifth}%
 Let $A$ and $B$ be $5\times 5$ matrices satisfying 
 \begin{equation*}
  \rank(A:B)=5 \quad \textrm{and}\quad AC_5A^*=BC_5B^*.
 \end{equation*}
 Then, there exist $2\times 2$ unitary matrices $U_1$ and $V_1$,
 $3\times 3$ unitary matrices $U_2$ and $V_2$
 and real diagonal $2\times 2$ matrices $C$ and $S$
 such that
 \begin{enumerate}[label=(\alph*)]
  \item
   $C^2+S^2=I_2$,
  \item
   $(A:B)$ has the canonical form
   \begin{equation*}
   (A:B) \sim \dfrac{1}{\sqrt{2}}Q_1
           \begin{pmatrix}
            C & 0 & I_2 & 0 & \mathbf{S} \\
            -\mathbf{S}^* & I_3 & 0 & I_3 & \mathbf{C}
           \end{pmatrix}
           Q_2,
   \end{equation*}
   where $(A:B) \sim (A':B')$ if and only if there exists $J$ such that $J^{-1}(A:B) = (A':B')$,
  \item $\rank(A) = 5-\Null(I_2-KK^*) = \rank(B)$,
  \item the boundary conditions are
   \begin{enumerate}[label=(\roman*)]
    \item mixed, if and only if $\rank(I_2-KK^*)<2$,
    \item coupled, if and only if $\rank(I_2-KK^*)=2$,
   \end{enumerate}
 \end{enumerate}
 \medskip
 where
 $\mathbf{S}=\begin{pmatrix} 0 & S \end{pmatrix}$,
 $\mathbf{C}=\begin{pmatrix}1&0\\0& C\end{pmatrix}$,
 $K=\begin{pmatrix}0&1&0\\0&0&1\end{pmatrix}U_2\mathbf{C},$
 \begin{gather*}
  Q_1 = \begin{pmatrix}
         U_1 & 0 \\
         0 & U_2
       \end{pmatrix}, \quad
  Q_2 = \begin{pmatrix}
         V_1 & 0 & 0 & 0 & 0 \\
           0 & U_2^* & 0 & 0 & 0 \\
           0 & 0 & U_1^* & 0 & 0 \\
           0 & 0 & 0 & U_2^* & 0  \\
         0 & 0 & 0 & 0 & V_2
        \end{pmatrix}
        Q_3, \\
  Q_3 = \begin{pmatrix}
         I_2 & 0 & 0 & 0 & 0 \\
           0 & \begin{bmatrix}1\\0\\0\end{bmatrix} & 0 & 0 & 0 \\
           0 & 0 & I_2 & 0 & 0 \\
           0 & 0 & 0 & \begin{bmatrix}0\\ I_2 \end{bmatrix} & 0  \\
           0 & 0 & 0 & 0 & I_3
        \end{pmatrix}Q_4,\quad
  Q_4 = \begin{pmatrix}
         I_2 &      0 &  C_2 & 0 & 0 & 0 \\
           0 & \sqrt2 &    0 & 0 & 0 & 0 \\
         I_2 &      0 & -C_2 & 0 & 0 & 0 \\
         0 & 0 & 0 & I_2 &      0 &  C_2 \\
         0 & 0 & 0 &   0 & \sqrt2 &    0 \\
         0 & 0 & 0 & I_2 &      0 & -C_2
        \end{pmatrix}.
 \end{gather*}
\end{lemma}

\begin{proof}
Equation \eqref{fiftheq3} can be written in the form
\begin{equation}
 \label{fiftheqdsum}
 \rank(A:B) = 5, \qquad
 (A:B)\begin{pmatrix}C_5 & 0 \\ 0 & -C_5\end{pmatrix}(A:B)^* = 0,
\end{equation}
where $C_5\oplus(-C_5)=\begin{pmatrix}C_5 & 0 \\ 0 & -C_5\end{pmatrix}$ is a Hermitian
matrix with eigenvalues $1$ and $-1$. Thus, each column vector $\mathbf{x}_j^*$ ($j\in\{1,\ldots,5\}$) of
$(A:B)^*$ may be written in the form
\begin{equation}
 \label{eq:imi}
 \mathbf{x}_j^* = \mathbf{x}_{j,1}^* + \mathbf{x}_{j,-1}^*
\end{equation}
where $\mathbf{x}_{j,\pm 1}^*$ belongs to the eigenspace corresponding
to the eigenvalue $\pm 1$ of $C_5\oplus(-C_5)$. The second equation of condition \eqref{fiftheqdsum} may now be
written
\begin{equation}
 \label{eq:eqinp}
 \mathbf{x}_{j,1}\mathbf{x}_{k,1}^* = \mathbf{x}_{j,-1}\mathbf{x}_{k,-1}^*.
\end{equation}
Taking $\mathbf{x}_{j,1}$ as the rows of $X_1$ and similarly for $X_{-1}$,
\eqref{eq:imi} may be summarized as
\begin{equation*}
 (A:B)=X_1+X_{-1}
\end{equation*}
and \eqref{eq:eqinp} as 
\begin{equation}
 \label{eq:eqinpsum}
 X_1X_1^*=X_{-1}X_{-1}^*.
\end{equation}
Now decompose
\begin{equation}
 \label{eq:decompi}
 \begin{pmatrix}C_5 & 0 \\ 0 & -C_5\end{pmatrix}
 =
 V\begin{pmatrix}-I_5 & 0 \\ 0 & I_5\end{pmatrix}V^*
\end{equation}
where $V$ is an arbitrary unitary matrix providing the diagonalization.
From the ordering of eigenvectors (columns of $V$) in \eqref{eq:decompi} and
the solution \eqref{eq:imi} in terms  of eigenvectors, the matrix $V$
may be chosen so that $(A:B)$ has the form
\begin{equation}
 \label{eq:ABcanon}
 (A:B) = (C : D)V^*.
\end{equation}
In particular, we will choose
\begin{equation}
 \label{eq:Vsimp}
 V =
 \begin{pmatrix}
   V_{-1} & \mathbf{v}_{-1} & 0  & V_1 & 0               & 0 \\
   0      & 0               & V_1& 0   & \mathbf{v}_{-1} & V_{-1}
 \end{pmatrix},
\end{equation}
where
\begin{equation*}
 V_{1} = \frac1{\sqrt2}
 \begin{pmatrix}
  1 & 0 \\
  0 & 1 \\
  0 & 0 \\
  0 & 1 \\
 -1 & 0
 \end{pmatrix}
  =
 \frac1{\sqrt2}
 \begin{pmatrix} I_2 \\ 0 \\ -C_2 \end{pmatrix},
\end{equation*}
\begin{equation*}
 V_{-1} = \frac1{\sqrt2}
 \begin{pmatrix}
  1 & 0 \\
  0 & 1 \\
  0 & 0 \\
  0 &-1 \\
  1 & 0
 \end{pmatrix}
  =
 \frac1{\sqrt2}
 \begin{pmatrix} I_2 \\ 0 \\ C_2 \end{pmatrix},
 \qquad
 \mathbf{v}_{-1} = \begin{pmatrix} 0 \\ 0 \\ 1 \\ 0 \\ 0 \end{pmatrix}
\end{equation*}
and
\begin{equation*}
 C_2 = \begin{pmatrix} 0 & -1 \\ 1 & 0 \end{pmatrix}.
\end{equation*}
Writing
\begin{equation*}
 V =
 \begin{pmatrix}
  \mathbf{y}_{1,-1}^* &
  \cdots &
  \mathbf{y}_{5,-1}^* &
  \mathbf{y}_{1,1}^* &
  \cdots &
  \mathbf{y}_{5,1}^*
 \end{pmatrix},%\qquad
% V^* =
% \begin{pmatrix}
%  \mathbf{y}_{1,-1} \\
%  \vdots \\
%  \mathbf{y}_{5,-1} \\
%  \mathbf{y}_{1,1} \\
%  \vdots \\
%  \mathbf{y}_{5,1} \\
% \end{pmatrix},
\end{equation*}
where $V$ is unitary and each $\mathbf{y}_{j,\pm 1}$ is an eigenvector corresponding
to the eigenvalue $\pm 1$. Since each row of $X_{\pm 1}$ is a linear combination
of $\mathbf{y}_{j,\pm1}$,
\begin{equation*}
 X_1
 =
 C
 \begin{pmatrix}
  \mathbf{y}_{1,1} \\
  \vdots \\
  \mathbf{y}_{5,1}
 \end{pmatrix},\qquad
 X_{-1}
 =
 D
 \begin{pmatrix}
  \mathbf{y}_{1,-1} \\
  \vdots \\
  \mathbf{y}_{5,-1}
 \end{pmatrix}
\end{equation*}
(for some $5\times 5$ matrices $C$ and $D$)
so that \eqref{eq:eqinpsum} becomes
\begin{equation*}
 CC^*=DD^*.
\end{equation*}
The singular value decompositions $C=U_C\Sigma_C V_C^*$ and $D=U_D\Sigma_D V_D^*$ yields from
\begin{equation*}
 \sigma_C^2=(U_C^*U_D)\Sigma_D^2(U_C^*U_D)^*
\end{equation*}
and by uniqueness of positive definite square roots,
\begin{equation*}
 \Sigma_C=(U_C^*U_D)\Sigma_D(U_C^*U_D)^*.
\end{equation*}
Hence,
\begin{align*}
 (A:B)&=(U_C\Sigma_C:U_D\Sigma_D)\begin{pmatrix}V_C& 0\\ 0 & V_D\end{pmatrix}^*V^* \\
      &=U_C(\Sigma_C:U_C^*U_D\Sigma_D)\begin{pmatrix}V_C& 0\\ 0 & V_D\end{pmatrix}^*V^* \\
      &=U_C(\Sigma_C:\Sigma_C)\begin{pmatrix}V_C& 0\\ 0 & V_D(U_C^*U_D)^*\end{pmatrix}^*V^* \\
      &=U_C\Sigma_C(I_5:I_5)\begin{pmatrix}V_C& 0\\ 0 & V_D(U_C^*U_D)^*\end{pmatrix}^*V^*
\end{align*}
yields the solution \eqref{eq:imi}  and satisfies \eqref{eq:eqinp}.
Since $\rank(A:B)=5$, we have $\rank(\Sigma_C)=5$ and hence $\Sigma_C$ is invertible.
By invariance of the boundary conditions under elementary row operations, we obtain
the general form
\begin{equation}
 \label{eq:gencanon}
 (A:B) \sim (I_5:I_5)\begin{pmatrix}V_X& 0\\ 0 & V_Y\end{pmatrix}^*V^*
\end{equation}
where $V_X$ and $V_Y$ are arbitrary unitary matrices. Here, the first
5 columns of $V$ are eigenvectors corresponding to the eigenvalue $-1$ of $C_5\oplus(-C_5)$,
and the remaining 5 columns correspond to the eigenvalue $1$.
We write $V$ as the block matrix
\begin{equation*}
 V = \begin{pmatrix} V_{11} & V_{12} \\ V_{21} & V_{22} \end{pmatrix}
\end{equation*}
so that \eqref{eq:gencanon} becomes
\begin{equation*}
 (A:B) \sim (V_X^*V_{11}^*+V_Y^*V_{12}^* : V_X^*V_{21}^* + V_Y^*V_{22}^*)
\end{equation*}
where
\begin{equation*}
 \begin{pmatrix}C_5 & 0 \\ 0 & -C_5\end{pmatrix}
 \begin{pmatrix}V_{11} \\ V_{21}\end{pmatrix}
 =
 -\begin{pmatrix}V_{11} \\ V_{21}\end{pmatrix},
 \qquad
 \begin{pmatrix}C_5 & 0 \\ 0 & -C_5\end{pmatrix}
 \begin{pmatrix}V_{12} \\ V_{22}\end{pmatrix}
 =
  \begin{pmatrix}V_{12} \\ V_{22}\end{pmatrix}.
\end{equation*}
Again, since the boundary conditions are invariant under row operations, we will assume
\begin{equation}
 \label{eq:finalcanon}
 (A:B) \sim (V_{11}^*+WV_{12}^* : V_{21}^* + WV_{22}^*)
\end{equation}
where $W=V_XV_Y^*$ is unitary.
From \eqref{eq:Vsimp}, we have
\begin{equation*}
% \label{eq:vblock}
 \begin{split}
  V_{11} = \frac1{\sqrt2}\begin{pmatrix} I_2 & 0 & 0_2 \\ 0 & \sqrt2 & 0 \\ C_2 & 0 & 0_2 \end{pmatrix},\qquad
  V_{12} = \frac1{\sqrt2}\begin{pmatrix} I_2 & 0 & 0_2 \\ 0 & 0 & 0 \\ -C_2 & 0 & 0_2 \end{pmatrix},\\
  V_{21} = \frac1{\sqrt2}\begin{pmatrix} 0_2 & 0 & I_2 \\ 0 & 0 & 0 \\ 0_2 & 0 & -C_2 \end{pmatrix},\qquad
  V_{22} = \frac1{\sqrt2}\begin{pmatrix} 0_2 & 0 & I_2 \\ 0 & \sqrt2 & 0 \\ 0_2 & 0 & C_2 \end{pmatrix}.
 \end{split}
\end{equation*}
Choosing appropriate $W$ provides the remaining canonical forms. Thus
\begin{align*} 
A & \sim \frac1{\sqrt2}\left[
       \begin{pmatrix}
        I_2 & 0      & -C_2 \\
          0 & \sqrt2 & 0  \\
          0 & 0      & 0
       \end{pmatrix}
       +
       W
       \begin{pmatrix}
        I_2 & 0 & C_2 \\
          0 & 0 & 0  \\
          0 & 0 & 0
       \end{pmatrix}
      \right], \\
B & \sim \frac1{\sqrt2}\left[
       \begin{pmatrix}
          0 & 0 & 0 \\
          0 & 0 & 0 \\
        I_2 & 0 & C_2
       \end{pmatrix}
       +
       W
       \begin{pmatrix}
          0 & 0      & 0 \\
          0 & \sqrt2 & 0 \\
        I_2 & 0      & -C_2
       \end{pmatrix}
      \right].
\end{align*}
Let
\begin{equation*}
 W = \begin{pmatrix} W_1 & \mathbf{w}_1 & W_2 \\ \mathbf{w}_2^T & w_3 & \mathbf{w}_4^T \\ W_3 & \mathbf{w}_5 & W_4 \end{pmatrix}
\end{equation*}
where $W_1, W_2,W_3,W_4 \in M_2(\mathbb{C})$, $\mathbf{w}_1, \mathbf{w}_2, \mathbf{w}_4, \mathbf{w}_5 \in\mathbb{C}^2$ and $w_3\in\mathbb{C}$.
It follows that
\begin{equation} 
\label{eq:factor1}
\begin{split}
A & \sim\frac1{\sqrt2}
       \begin{pmatrix}
        I_2+W_1 & 0 & -C_2+W_1C_2 \\
        \mathbf{w}_2^T & \sqrt{2} & \mathbf{w}_2^TC_2  \\
        W_3 & 0 & W_3C_2
       \end{pmatrix} \\
  &   = \frac1{\sqrt2}
        \begin{pmatrix} W_1 & 0 & I_2 \\ \mathbf{w}_2^T & 1 & 0 \\ W_3 & 0 & 0 \end{pmatrix}
        \begin{pmatrix}
         I_2 &      0 &  C_2 \\
           0 & \sqrt2 &    0 \\
         I_2 &      0 & -C_2
        \end{pmatrix}, \\
B & \sim\frac1{\sqrt2}
       \begin{pmatrix}
        W_2 & \sqrt2\mathbf{w}_1 & W_2C_2 \\
        \mathbf{w}_4^T & \sqrt2 w_3 & \mathbf{w}_4^TC_2  \\
        I_2+W_4 & \sqrt2\mathbf{w}_5 & C_2-W_4C_2
       \end{pmatrix} \\
  &   = \frac1{\sqrt2}
        \begin{pmatrix} 0 & \mathbf{w}_1 & W_2 \\ 0 & w_3 & \mathbf{w}_4^T \\ I_2 & \mathbf{w}_5 & W_4 \end{pmatrix}
        \begin{pmatrix}
         I_2 &      0 &  C_2 \\
           0 & \sqrt2 &    0 \\
         I_2 &      0 & -C_2
        \end{pmatrix},
\end{split}
\end{equation}
and hence,
\begin{align}
 \rank(A) &= 3+\rank(W_3), \label{eq:rankA5} \\
 \rank(B) &= 2+\rank\begin{pmatrix}\mathbf{w}_1 & W_2 \\ w_3 & \mathbf{w}_4^T \end{pmatrix}, \label{eq:rankB5}
\end{align}
The CS-decomposition, described in detail in \cite{paige94} and \cite[Theorem 2.7.1]{horn12},
provides a useful way to speak about rank. In particular, we obtain using \cite[Theorem 4.1]{fuhr18}
\begin{equation*}
 W = \begin{pmatrix} G_1 & 0 \\ 0 & G_2 \end{pmatrix}
     \begin{pmatrix} (I_2+F)/2 & 0 & (I_2-F)/2 \\ 0 & 1 & 0 \\ (I_2-F)/2 & 0 & (I_2+F)/2 \end{pmatrix}
     \begin{pmatrix} G_3 & 0 \\ 0 & G_4 \end{pmatrix}
\end{equation*}
for some $2\times 2$ unitary matrices $F$, $G_1$, $G_3$, and $3\times 3$ unitary matrices $G_2$ and $G_4$.
Since
\begin{equation*}
 \begin{pmatrix} (I_2+F)/2 & (I_2-F)/2 \\ (I_2-F)/2 & (I_2+F)/2 \end{pmatrix}
\end{equation*}
is a $4\times 4$ unitary matrix, we can further apply \cite[Corollary 3.1]{fuhr18} and express $W$ by
\begin{equation*}
 W = \begin{pmatrix} U_1 & 0 \\ 0 & U_2 \end{pmatrix}
     \begin{pmatrix} C & 0 & S \\ 0 & 1 & 0 \\ -S & 0 & C \end{pmatrix}
     \begin{pmatrix} V_1 & 0 \\ 0 & V_2 \end{pmatrix}
\end{equation*}
for some $2\times 2$ unitary matrices $U_1$, $V_1$, and $3\times 3$ unitary matrices $U_2$ and $V_2$
and positive semi-definite diagonal matrices $C$ and $S$ satisfying $C^2+S^2=I_2$.
Consequently,
\begin{equation*}
 W = \begin{pmatrix}
       U_1CV_1 & U_1\begin{bmatrix}0 & S\end{bmatrix}V_2 \\
       U_2\begin{bmatrix} 0 \\ -S\end{bmatrix}V_1 & U_2\begin{bmatrix} 1 & 0 \\ 0 &  C\end{bmatrix}V_2
      \end{pmatrix}
\end{equation*}
so that, using $\mathbf{S}=\begin{bmatrix}0&S\end{bmatrix}$ and $\mathbf{C}=1\oplus C$,
\begin{equation}
 \label{eq:W5}
 \begin{pmatrix} W_1 & \mathbf{w}_1 & W_2 \\ \mathbf{w}_2^T & w_3 & \mathbf{w}_4^T \\ W_3 & \mathbf{w}_5 & W_4 \end{pmatrix}
 =
 \begin{pmatrix}
  U_1CV_1 & U_1\mathbf{S}V_2 \\
  -U_2\mathbf{S}^*V_1 & U_2\mathbf{C}V_2
 \end{pmatrix}
\end{equation}
Hence, \eqref{eq:factor1} gives
\begin{equation*}
 (A:B) \sim
 \dfrac{1}{\sqrt{2}}
 \begin{pmatrix}
  W_1 & 0 & I_2
  & 0 & \mathbf{w}_1 & W_2 \\
  \mathbf{w}_2^T & 1 & 0
  & 0 & w_3 & \mathbf{w}_4^T \\
  W_3 & 0 & 0
  & I_2 & \mathbf{w}_5 & W_4
 \end{pmatrix} Q_4
\end{equation*}
where
\begin{equation*}
Q_4 = 
 \begin{pmatrix}
  I_2 &      0 &  C_2 & 0 & 0 & 0 \\
    0 & \sqrt2 &    0 & 0 & 0 & 0 \\
  I_2 &      0 & -C_2 & 0 & 0 & 0 \\
  0 & 0 & 0 & I_2 &      0 &  C_2 \\
  0 & 0 & 0 &   0 & \sqrt2 &    0 \\
  0 & 0 & 0 & I_2 &      0 & -C_2
 \end{pmatrix}.
\end{equation*}
It follows that
\begin{align*}
 (A:B) & \sim
 \dfrac{1}{\sqrt{2}}
 \begin{pmatrix}
  U_1CV_1 & \begin{bmatrix} 0 & I_2 & 0_2\end{bmatrix} & U_1\mathbf{S}V_2 \\
  -U_2\mathbf{S}^*V_1 & \begin{bmatrix} 1 & 0 & 0 \\ 0 & 0 & I_2 \end{bmatrix} & U_2\mathbf{C}V_2
 \end{pmatrix} Q_4 \\
 &=
 \dfrac{1}{\sqrt{2}}
 \begin{pmatrix}
  U_1 & 0 \\
    0 & U_2
 \end{pmatrix}
 \begin{pmatrix}
  C & \begin{bmatrix} 0 & U_1^* & 0_2\end{bmatrix} & \mathbf{S} \\
  -\mathbf{S}^* & U_2^*\begin{bmatrix} 1 & 0 & 0 \\ 0 & 0 & I_2 \end{bmatrix} & \mathbf{C}
 \end{pmatrix}
 \begin{pmatrix}
  V_1 & 0 & 0 \\
    0 & I_5 & 0  \\
  0 & 0 & V_2
 \end{pmatrix} Q_4 \\
 &=
 \dfrac{1}{\sqrt{2}}
 \begin{pmatrix}
  U_1 & 0 \\
    0 & U_2
 \end{pmatrix}
 \begin{pmatrix}
  C & 0 & I_2 & 0 & \mathbf{S} \\
  -\mathbf{S}^* & I_3 & 0 & I_3 & \mathbf{C}
 \end{pmatrix} \\
 & \qquad\times
 \begin{pmatrix}
  V_1 & 0 & 0 & 0 & 0 \\
    0 & U_2^*\begin{bmatrix}1\\0\\0\end{bmatrix} & 0 & 0 & 0 \\
    0 & 0 & U_1^* & 0 & 0 \\
    0 & 0 & 0 & U_2^*\begin{bmatrix}0\\ I_2 \end{bmatrix} & 0  \\
  0 & 0 & 0 & 0 & V_2
 \end{pmatrix} Q_4 \\
 &=
 \dfrac{1}{\sqrt{2}}
 \begin{pmatrix}
  U_1 & 0 \\
    0 & U_2
 \end{pmatrix}
 \begin{pmatrix}
  C & 0 & I_2 & 0 & \mathbf{S} \\
  -\mathbf{S}^* & I_3 & 0 & I_3 & \mathbf{C}
 \end{pmatrix}
 \begin{pmatrix}
  V_1 & 0 & 0 & 0 & 0 \\
    0 & U_2^* & 0 & 0 & 0 \\
    0 & 0 & U_1^* & 0 & 0 \\
    0 & 0 & 0 & U_2^* & 0  \\
  0 & 0 & 0 & 0 & V_2
 \end{pmatrix}
 \\
 &\qquad \times
 \begin{pmatrix}
  I_2 & 0 & 0 & 0 & 0 \\
    0 & \begin{bmatrix}1\\0\\0\end{bmatrix} & 0 & 0 & 0 \\
    0 & 0 & I_2 & 0 & 0 \\
    0 & 0 & 0 & \begin{bmatrix}0\\ I_2 \end{bmatrix} & 0  \\
  0 & 0 & 0 & 0 & I_3
 \end{pmatrix} Q_4.
\end{align*}
It follows from \eqref{eq:rankA5}, \eqref{eq:rankB5} and \eqref{eq:W5} that
\begin{align*}
 \rank(A)&=3+\rank\left(\begin{bmatrix} 0 & 1 & 0 \\ 0 & 0 & 1 \end{bmatrix}
                         U_2(-\mathbf{S}^*)V_1\right)
          =3+\rank\left(\begin{bmatrix} 0 & 1 & 0 \\ 0 & 0 & 1 \end{bmatrix}U_2\mathbf{S}^*\right), \\
 \rank(B)&=2+\rank\begin{pmatrix}
               U_1\mathbf{S}V_2 \\
               \begin{bmatrix}1&0&0\end{bmatrix}U_2\mathbf{C}V_2
             \end{pmatrix}
          =2+\rank\begin{pmatrix}
                   \mathbf{S} \\
                   \begin{bmatrix}1&0&0\end{bmatrix}U_2\mathbf{C}
                  \end{pmatrix}.
\end{align*}
We note that $\rank(A)=\rank(AA^*)=\rank(A^*A)$ and similarly for $B$.
Using $\mathbf{S}^*\mathbf{S}=I_3-\mathbf{C}^*\mathbf{C}=I_3-\mathbf{C}^2$, we find
\begin{align*}
  \left(\begin{bmatrix} 0 & 1 & 0 \\ 0 & 0 & 1 \end{bmatrix}U_2\mathbf{S}^*\right)
  \left(\begin{bmatrix} 0 & 1 & 0 \\ 0 & 0 & 1 \end{bmatrix}U_2\mathbf{S}^*\right)^*
 &= \begin{pmatrix} 0 & 1 & 0 \\ 0 & 0 & 1 \end{pmatrix}
    U_2\mathbf{S}^*\mathbf{S}U_2^*
    \begin{pmatrix} 0 & 0 \\ 1 & 0 \\ 0 & 1 \end{pmatrix}\\
 &= I_2
  - \begin{pmatrix} 0 & 1 & 0 \\ 0 & 0 & 1 \end{pmatrix}
    U_2\mathbf{C}^2U_2^*
    \begin{pmatrix} 0 & 0 \\ 1 & 0 \\ 0 & 1 \end{pmatrix} \\
 \begin{pmatrix} \mathbf{S} \\ \begin{bmatrix}1&0&0\end{bmatrix}U_2\mathbf{C} \end{pmatrix}^*
 \begin{pmatrix} \mathbf{S} \\ \begin{bmatrix}1&0&0\end{bmatrix}U_2\mathbf{C} \end{pmatrix}
  &= \mathbf{S}^*\mathbf{S}
   + \mathbf{C}^*U_2^*
     \begin{pmatrix}1&0&0\\0&0&0\\0&0&0\end{pmatrix}
     U_2\mathbf{C} \\
  &= I_3-\mathbf{C}^*U_2^*
     \begin{pmatrix}0&0&0\\0&1&0\\0&0&1\end{pmatrix}
     U_2\mathbf{C} \\
  &= I_3-\mathbf{C}U_2^*
     \begin{pmatrix} 0 & 0 \\ 1 & 0 \\ 0 & 1 \end{pmatrix}
     \begin{pmatrix} 0 & 1 & 0 \\ 0 & 0 & 1 \end{pmatrix}
     U_2\mathbf{C}.
\end{align*}
Hence,
\begin{align*}
 \rank(A)&=3+
  \rank\left(I_2 - KK^*\right), \\
 \rank(B)&=2+
  \rank\left(I_3 - K^*K\right).
\end{align*}
where
\begin{equation*}
 K=\begin{pmatrix}0&1&0\\0&0&1\end{pmatrix}U_2\mathbf{C}.
\end{equation*}
Since $K^*K$ and $KK^*$ have the same non-zero eigenvalues with the
same multiplicities $KK^*$ and $K^*K$ have the same number of
unit eigenvalues (i.e. equal to 1) so $\Null(I_2-KK^*)=\Null(I_3-K^*K)$.
By the rank-nullity duality
\begin{equation*}
 \rank(A) = 3+2-\Null(I_2-KK^*) = 2+3-\Null(I_3-K^*K) = \rank(B).
 \qedhere
\end{equation*}
\end{proof}

\section{Canonical forms for $2n+1$-th order differential operators}\label{cano2n+1}

We consider on the interval $J=(a,b)$, $-\infty\le a<b\le \infty$, the $2n+1$-th order differential equation 
with formally self-adjoint differential expression (with smooth coefficients) \cite{hinton78,lin23}
\begin{equation}\label{eq:2n+1th}
 My=\sum_{k=0}^n (-1)^k\left(i\left[\left(q_{n-k}y^{(k)}\right)^{(k+1)}
   +\left(q_{n-k}y^{(k+1)}\right)^{(k)}\right]+\left(p_{n-k}y^{(k)}\right)^{(k)}\right)=\lambda wy,
\end{equation}
where $q_0^{-1}$ exists on $J$, $p_j,q_j\in C^j(J)$ are sufficiently smooth real-valued functions on $J$ and
$w\in L(J,\mathbb R)$ is a real-valued Lebesgue integrable function on $J$, $w>0$ a.e. on $J$.
Let
\begin{equation*}
 \theta:=\dfrac1{\sqrt2}+\dfrac{i}{\sqrt2}.
\end{equation*}
If the coefficients in \eqref{eq:2n+1th} are not smooth, we introduce the quasi-derivatives
of $y$ given by
\begin{equation*}
 y^{[k]}=y^{(k)},\qquad 0\leq k\leq n-1,
\end{equation*}
and
\begin{align*}
% &y^{[k]}=y^{(k)}, && \text{if $0\leq k\leq n-1$} \\
 &y^{[n]}=-\theta q_0(y^{[n-1]})', \\
 &y^{[n+1]} = -\theta q_0(y^{[n]})'+i\theta{p_0}{q_0}^{-1}y^{[n]}-iq_1y^{[n-1]}, \\
 &y^{[n+2]} = -(y^{[n+1]})'-\theta{q_1}{q_0}^{-1}y^{[n]}+p_1y^{[n-1]}-iq_2y^{[n-2]}, && \text{if $n\neq 1$} \\
% &y^{[n+k+1]} = -(y^{[n+k]})'+p_ky^{[n-k]}+i(q_ky^{[n-k+1]}-q_{k+1}y^{[n-k-1]}), && \text{if $2\leq k\leq n-1$}\\
 &y^{[n+k]} = -(y^{[n+k-1]})'+p_{k-1}y^{[n-k+1]}+i(q_{k-1}y^{[n-k+2]}-q_{k}y^{[n-k]}), && \text{if $3\leq k\leq n$}\\
 &y^{[2n+1]} = -(y^{[2n]})'-iq_ny'+p_ny.
\end{align*}
Observe that $n+2=2n+1$ if and only if $n=1$ if and only if $n+2>2n$. Hence, the condition $n\neq 1$
is written as $n+2\leq 2n$ in \cite{hinton78,lin23}.
In either case, the boundary conditions have the same form.
Let $Y=\left(y, y', y'', y^{[3]}, y^{[4]},\ldots,y^{[2n]}\right)^{\top}$.
We now consider the $2n+1$-th order boundary value problem defined by \eqref{eq:2n+1th} and the boundary conditions
\begin{align}
 \label{2n+1theq2}
 AY(a)+BY(b)=0, \quad A, B\in M_{2n+1}(\mathbb C).
\end{align}
For the boundary conditions \eqref{2n+1theq2} with the assumptions made so far, \cite[Theorem 2.4]{MolZet2} leads to  
\begin{proposition}%
 \label{2n+1prop1}%
 Let $C_{2n+1}$ be the symplectic matrix of order $2n+1$ defined by
 \begin{align}\label{2n+1theq4}
  C_{2n+1}=\left((-1)^r\delta_{r,2n+2-s}\right)_{r,s=1}^{2n+1},
 \end{align}
 where $\delta$ is the Kronecker delta.
 Then problems \eqref{eq:2n+1th}--\eqref{2n+1theq2} are self-adjoint if and only if 
 \begin{align}\label{2n+1theq3}
  \rank(A:B)=2n+1 \quad \textrm{and}\quad AC_{2n+1}A^*=BC_{2n+1}B^*.
 \end{align}
\end{proposition}

Lemma \ref{lem:fifth} and its proof provide the inspiration
for the following theorems and their proof.

\begin{theorem}
 \label{thm:2n+1th-odd}%
 Let $n\in\mathbb{N}$ be odd and let $A$ and $B$ be $(2n+1)\times (2n+1)$ matrices satisfying 
 \begin{equation*}
  \rank(A:B)=2n+1 \quad \textrm{and}\quad AC_{2n+1}A^*=BC_{2n+1}B^*.
 \end{equation*}
 Then, there exist $n\times n$ unitary matrices $U_1$ and $V_1$,
 $(n+1)\times (n+1)$ unitary matrices $U_2$ and $V_2$
 and real diagonal $n\times n$ matrices $C$ and $S$
 such that
 \begin{enumerate}[label=(\alph*)]
  \item
   $C^2+S^2=I_n$,
  \item
   $(A:B)$ has the canonical form
   \begin{equation*}
   (A:B) \sim \dfrac{1}{\sqrt{2}} Q_1
           \begin{pmatrix}
            \mathbf{C} & I_{n+1} & 0 & I_{n+1} & \mathbf{S} \\
            -\mathbf{S}^* & 0 & I_n & 0 & C
           \end{pmatrix}
           Q_2,
   \end{equation*}
   where $(A:B) \sim (A':B')$ if and only if there exists $J$ such that $J^{-1}(A:B) = (A':B')$,
  \item $\rank(A) = 2n+1-\Null(I_n-KK^*) = \rank(B)$,
  \item the boundary conditions are
   \begin{enumerate}[label=(\roman*)]
    \item mixed, if and only if $\rank(I_n-KK^*)<n$,
    \item coupled, if and only if $\rank(I_n-KK^*)=n$,
   \end{enumerate}
 \end{enumerate}
 \medskip
 where
 $\mathbf{S}=\begin{pmatrix} S & 0 \end{pmatrix}^*$,
 $\mathbf{C}=C\oplus 1=\begin{pmatrix}C&0\\0&1\end{pmatrix}$,
 $K=\begin{pmatrix}I_n&0\end{pmatrix}U_1\mathbf{C},$
 \begin{gather*}
  Q_1 = \begin{pmatrix}
         U_1 & 0 \\
           0 & U_2
        \end{pmatrix},\quad
  Q_2 = \begin{pmatrix}
         V_1 &     0 &     0 &     0 &   0 \\
           0 & U_1^* &     0 &     0 &   0 \\
           0 &     0 & U_2^* &     0 &   0 \\
           0 &     0 &     0 & U_1^* &   0  \\
           0 &     0 &     0 &     0 & V_2
        \end{pmatrix}
        Q_3, \\
  Q_3 = \begin{pmatrix}
         I_{n+1} & 0 & 0 & 0 & 0 \\
           0 & \begin{bmatrix} I_n \\ 0 \end{bmatrix} & 0 & 0 & 0 \\
           0 & 0 & I_n & 0 & 0 \\
           0 & 0 & 0 & \begin{bmatrix}0\\\vdots \\0\\1\end{bmatrix} & 0  \\
           0 & 0 & 0 & 0 & I_n
        \end{pmatrix}Q_4, \\
  Q_4 = \begin{pmatrix}
         I_n &      0 & (-1)^{n+1}C_n^* & 0 & 0 & 0 \\
           0 & \sqrt2 &             0 & 0 & 0 & 0 \\
         I_n &      0 &   (-1)^{n}C_n^* & 0 & 0 & 0 \\
         0 & 0 & 0 & I_n &      0 &  (-1)^{n+1}C_n^* \\
         0 & 0 & 0 &   0 & \sqrt2 &              0 \\
         0 & 0 & 0 & I_n &      0 &    (-1)^{n}C_n^*
        \end{pmatrix}.
 \end{gather*}
\end{theorem}

\begin{theorem}
 \label{thm:2n+1th-even}%
 Let $n\in\mathbb{N}$ be even and let $A$ and $B$ be $(2n+1)\times (2n+1)$ matrices satisfying 
 \begin{equation*}
  \rank(A:B)=2n+1 \quad \textrm{and}\quad AC_{2n+1}A^*=BC_{2n+1}B^*.
 \end{equation*}
 Then, there exist $(n+1)\times(n+1)$ unitary matrices $U_1$ and $V_1$,
 $n\times n$ unitary matrices $U_2$ and $V_2$
 and real diagonal $n\times n$ matrices $C$ and $S$
 such that
 \begin{enumerate}[label=(\alph*)]
  \item
   $C^2+S^2=I_n$,
  \item
   $(A:B)$ has the canonical form
   \begin{equation*}
   (A:B) \sim \dfrac{1}{\sqrt{2}} Q_1
           \begin{pmatrix}
            C & 0 & I_n & 0 & \mathbf{S} \\
            -\mathbf{S}^* & I_{n+1} & 0 & I_{n+1} & \mathbf{C}
           \end{pmatrix}
           Q_2,
   \end{equation*}
   where $(A:B) \sim (A':B')$ if and only if there exists $J$ such that $J^{-1}(A:B) = (A':B')$,
  \item $\rank(A) = 2n+1-\Null(I_n-KK^*) = \rank(B)$,
  \item the boundary conditions are
   \begin{enumerate}[label=(\roman*)]
    \item mixed, if and only if $\rank(I_n-KK^*)<n$,
    \item coupled, if and only if $\rank(I_n-KK^*)=n$,
   \end{enumerate}
 \end{enumerate}
 \medskip
 where
 $\mathbf{S}=\begin{pmatrix} 0 & S \end{pmatrix}$,
 $\mathbf{C}=1\oplus C=\begin{pmatrix}1&0\\0&C\end{pmatrix}$,
 $K=\begin{pmatrix}0&I_n\end{pmatrix}U_2\mathbf{C},$
 \begin{gather*}
  Q_1 = \begin{pmatrix}
         U_1 & 0 \\
           0 & U_2
        \end{pmatrix},\quad
  Q_2 = \begin{pmatrix}
         V_1 &     0 &     0 &     0 &   0 \\
           0 & U_2^* &     0 &     0 &   0 \\
           0 &     0 & U_1^* &     0 &   0 \\
           0 &     0 &     0 & U_2^* &   0  \\
           0 &     0 &     0 &     0 & V_2
        \end{pmatrix}
        Q_3, \\
  Q_3 = \begin{pmatrix}
         I_n & 0 & 0 & 0 & 0 \\
           0 & \begin{bmatrix}1\\0\\ \vdots \\0\end{bmatrix} & 0 & 0 & 0 \\
           0 & 0 & I_n & 0 & 0 \\
           0 & 0 & 0 & \begin{bmatrix}0\\ I_n \end{bmatrix} & 0  \\
          0 & 0 & 0 & 0 & I_{n+1}
        \end{pmatrix}Q_4, \\
  Q_4 = \begin{pmatrix}
         I_n &      0 & (-1)^{n+1}C_n^* & 0 & 0 & 0 \\
           0 & \sqrt2 &             0 & 0 & 0 & 0 \\
         I_n &      0 &   (-1)^{n}C_n^* & 0 & 0 & 0 \\
         0 & 0 & 0 & I_n &      0 &  (-1)^{n+1}C_n^* \\
         0 & 0 & 0 &   0 & \sqrt2 &              0 \\
         0 & 0 & 0 & I_n &      0 &    (-1)^{n}C_n^*
        \end{pmatrix}.
 \end{gather*}
\end{theorem}

\section{Proof of the main theorems}

\begin{proof}[Proof of Theorem \ref{thm:2n+1th-odd}]
Equation \eqref{2n+1theq3} can be written in the form
\begin{equation}
 \label{2n+1eqdsum-odd}
 \rank(A:B) = 2n+1, \qquad
 (A:B)\begin{pmatrix}C_{2n+1} & 0 \\ 0 & -C_{2n+1}\end{pmatrix}(A:B)^* = 0,
\end{equation}
where $\begin{pmatrix}C_{2n+1} & 0 \\ 0 & -C_{2n+1}\end{pmatrix}$ is a Hermitian
matrix with eigenvalues $1$ and $-1$. Thus, each column vector $\mathbf{x}_j^*$ of
$(A:B)^*$ may be written in the form
\begin{equation}
 \label{eq:imigen-odd}
 \mathbf{x}_j^* = \mathbf{x}_{j,1}^* + \mathbf{x}_{j,-1}^*
\end{equation}
where $\mathbf{x}_{j,\pm 1}^*$ belongs to the eigenspace corresponding
to the eigenvalue $\pm 1$. Condition \eqref{2n+1eqdsum-odd} may now be
written
\begin{equation}
 \label{eq:eqinpgen-odd}
 \mathbf{x}_{j,1}\mathbf{x}_{k,1}^* = \mathbf{x}_{j,-1}\mathbf{x}_{k,-1}^*.
\end{equation}
Taking $\mathbf{x}_{j,1}$ as the rows of $X_1$ and similarly for $X_{-1}$,
\eqref{eq:imigen-odd} may be summarized as $(A:B)=X_1+X_{-1}$ and \eqref{eq:eqinpgen-odd}
as 
\begin{equation}
 \label{eq:eqinpgensum-odd}
 X_1X_1^*=X_{-1}X_{-1}^*.
\end{equation}
Now let $V$ be a matrix that diagonalizes $C_{2n+1}\oplus(-C_{2n+1})$ as follows:
\begin{equation}
 \label{eq:decompigen-odd}
 \begin{pmatrix}C_{2n+1} & 0 \\ 0 & -C_{2n+1}\end{pmatrix}
 =
 V\begin{pmatrix}-I_{2n+1} & 0 \\ 0 & I_{2n+1}\end{pmatrix}V^*.
\end{equation}
From the ordering of eigenvectors (i.e. the columns of $V$) in \eqref{eq:decompigen-odd} and
the solution \eqref{eq:imigen-odd} given in terms of eigenvectors, the matrix $V$
may be chosen so that $(A:B)$ has the form
\begin{equation}
 \label{eq:ABcanon-odd}
 (A:B) = (C : D)V^*.
\end{equation}
In particular, we will choose
\begin{equation}
 \label{eq:Vsimp2n+1-odd}
 V =
 \begin{pmatrix}
   V_{-1} & 0              & 0  & V_1 & \mathbf{v}_{1}  & 0 \\
   0      & \mathbf{v}_{1} & V_1& 0   & 0              & V_{-1}
 \end{pmatrix},
\end{equation}
where
\begin{equation}
 \label{eq:V1}%
 V_{1} =
 \frac1{\sqrt2}
 \begin{pmatrix} I_n \\ 0 \\ (-1)^{n+1}C_n \end{pmatrix},
 \qquad
 \mathbf{v}_{1} =
  \underbrace{
   \begin{pmatrix}
   0 &
   \cdots &
   0 &
   1 &
   0 &
   \cdots &
   0
   \end{pmatrix}^T
   \!\!\!
  }_{\text{$(n+1)$-th coordinate}}
\end{equation}
\begin{equation}
 \label{eq:V-1}%
 V_{-1} =
 \frac1{\sqrt2}
 \begin{pmatrix} I_n \\ 0 \\ (-1)^nC_n \end{pmatrix},
\end{equation}
and
\begin{equation*}
 C_{n}=\left((-1)^r\delta_{r,n+1-s}\right)_{r,s=1}^{n}.
\end{equation*}
Writing
\begin{equation*}
 V =
 \begin{pmatrix}
  \mathbf{y}_{1,-1}^* &
  \cdots &
  \mathbf{y}_{2n+1,-1}^* &
  \mathbf{y}_{1,1}^* &
  \cdots &
  \mathbf{y}_{2n+1,1}^*
 \end{pmatrix}%,\qquad
 %V^* =
 %\begin{pmatrix}
 % \mathbf{y}_{1,-1} \\
 % \vdots \\
 % \mathbf{y}_{2n+1,-1} \\
 % \mathbf{y}_{1,1} \\
 % \vdots \\
 % \mathbf{y}_{2n+1,1} \\
 %\end{pmatrix},
\end{equation*}
where $V$ is unitary and each $\mathbf{y}_{j,\pm 1}$ is an eigenvector corresponding
to the eigenvalue $\pm 1$. Since each row of $X_{\pm 1}$ is a linear combination
of $\mathbf{y}_{j,\pm1}$,
\begin{equation*}
 X_1
 =
 C
 \begin{pmatrix}
  \mathbf{y}_{1,1} \\
  \vdots \\
  \mathbf{y}_{2n+1,1}
 \end{pmatrix},\qquad
 X_{-1}
 =
 D
 \begin{pmatrix}
  \mathbf{y}_{1,-1} \\
  \vdots \\
  \mathbf{y}_{2n+1,-1}
 \end{pmatrix}
\end{equation*}
(for some $(2n+1)\times(2n+1)$ matrices $C$ and $D$)
so that \eqref{eq:eqinpgensum-odd} becomes
\begin{equation*}
 CC^*=DD^*.
\end{equation*}
The singular value decompositions $C=U_C\Sigma_C V_C^*$ and $D=U_D\Sigma_D V_D^*$ yields from
\begin{equation*}
 \sigma_C^2=(U_C^*U_D)\Sigma_D^2(U_C^*U_D)^*
\end{equation*}
and by uniqueness of positive definite square roots,
\begin{equation*}
 \Sigma_C=(U_C^*U_D)\Sigma_D(U_C^*U_D)^*.
\end{equation*}
Hence,
\begin{align*}
 (A:B)&=(U_C\Sigma_C:U_D\Sigma_D)\begin{pmatrix}V_C& 0\\ 0 & V_D\end{pmatrix}^*V^* \\
      &=U_C(\Sigma_C:U_C^*U_D\Sigma_D)\begin{pmatrix}V_C& 0\\ 0 & V_D\end{pmatrix}^*V^* \\
      &=U_C(\Sigma_C:\Sigma_C)\begin{pmatrix}V_C& 0\\ 0 & V_D(U_C^*U_D)^*\end{pmatrix}^*V^* \\
      &=U_C\Sigma_C(I_{2n+1}:I_{2n+1})\begin{pmatrix}V_C& 0\\ 0 & V_D(U_C^*U_D)^*\end{pmatrix}^*V^*
\end{align*}
yields the solution \eqref{eq:imigen-odd}  and satisfies \eqref{eq:eqinpgen-odd}.
Since $\rank(A:B)=2n+1$, we have $\rank(\Sigma_C)=2n+1$ and hence $\Sigma_C$ is invertible.
By invariance of the boundary conditions under elementary row operations, we obtain
the general form
\begin{equation}
 \label{eq:gencanon2n+1-odd}
 (A:B) \sim (I_{2n+1}:I_{2n+1})\begin{pmatrix}V_X& 0\\ 0 & V_Y\end{pmatrix}^*V^*
\end{equation}
where $V_X$ and $V_Y$ are arbitrary unitary matrices. Here, the first
$2n+1$ columns of $V$ are eigenvectors corresponding to the eigenvalue $-1$ of $C_{2n+1}\oplus(-C_{2n+1})$,
and the remaining $2n+1$ columns correspond to the eigenvalue $1$.
We write $V$ as the block matrix
\begin{equation*}
 V = \begin{pmatrix} V_{11} & V_{12} \\ V_{21} & V_{22} \end{pmatrix}
\end{equation*}
so that \eqref{eq:gencanon2n+1-odd} becomes
\begin{equation*}
 (A:B) \sim (V_X^*V_{11}^*+V_Y^*V_{12}^* : V_X^*V_{21}^* + V_Y^*V_{22}^*)
\end{equation*}
where
\begin{equation*}
 \begin{pmatrix}C_{2n+1} & 0 \\ 0 & -C_{2n+1}\end{pmatrix}
 \begin{pmatrix}V_{11} \\ V_{21}\end{pmatrix}
 =
 -\begin{pmatrix}V_{11} \\ V_{21}\end{pmatrix},
 \qquad
 \begin{pmatrix}C_{2n+1} & 0 \\ 0 & -C_{2n+1}\end{pmatrix}
 \begin{pmatrix}V_{12} \\ V_{22}\end{pmatrix}
 =
  \begin{pmatrix}V_{12} \\ V_{22}\end{pmatrix}.
\end{equation*}
Again, since the boundary conditions are invariant under row operations, we will assume
\begin{equation}
 \label{eq:finalcanon-odd}
 (A:B) \sim (V_{11}^*+WV_{12}^* : V_{21}^* + WV_{22}^*)
\end{equation}
where $W=V_XV_Y^*$ is unitary.
From \eqref{eq:Vsimp2n+1-odd}, we have
\begin{align*}
% \label{eq:vblock2n+1-odd}
  V_{11} &= \frac1{\sqrt2}\begin{pmatrix} I_n & 0 & 0_n \\ 0 & 0 & 0 \\ (-1)^{n}C_n & 0 & 0_n \end{pmatrix},\\
  V_{12} &= \frac1{\sqrt2}\begin{pmatrix} I_n & 0 & 0_n \\ 0 & \sqrt2 & 0 \\ (-1)^{n+1}C_n & 0 & 0_n \end{pmatrix},\\
  V_{21} &= \frac1{\sqrt2}\begin{pmatrix} 0_n & 0 & I_n \\ 0 & \sqrt2 & 0 \\ 0_n & 0 & (-1)^{n+1}C_n \end{pmatrix},\\
  V_{22} &= \frac1{\sqrt2}\begin{pmatrix} 0_n & 0 & I_n \\ 0 & 0 & 0 \\ 0_n & 0 & (-1)^nC_n \end{pmatrix}.
\end{align*}
Choosing appropriate $W$ provides the remaining canonical forms. Thus
\begin{align*} 
A & \sim \frac1{\sqrt2}\left[
       \begin{pmatrix}
        I_n & 0            & (-1)^nC_n^* \\
          0 & 0            & 0  \\
        0_n & 0            & 0_n
       \end{pmatrix}
       +
       W
       \begin{pmatrix}
        I_n & 0            & (-1)^{n+1}C_n^* \\
          0 & \sqrt2       & 0  \\
        0_n & 0            & 0_n
       \end{pmatrix}
      \right], \\
B & \sim \frac1{\sqrt2}\left[
       \begin{pmatrix}
        0_n & 0            & 0_n \\
          0 & \sqrt2       & 0 \\
        I_n & 0            & (-1)^{n+1}C_n^*
       \end{pmatrix}
       +
       W
       \begin{pmatrix}
        0_n & 0            & 0_n \\
          0 & 0            & 0 \\
        I_n & 0            & (-1)^nC_n^*
       \end{pmatrix}
      \right],
\end{align*}
Let
\begin{equation*}
 W = \begin{pmatrix} W_1 & \mathbf{w}_1 & W_2 \\ \mathbf{w}_2^T & w_3 & \mathbf{w}_4^T \\ W_3 & \mathbf{w}_5 & W_4 \end{pmatrix}
\end{equation*}
where $W_1, W_2,W_3,W_4 \in M_n(\mathbb{C})$, $\mathbf{w}_1, \mathbf{w}_2, \mathbf{w}_4, \mathbf{w}_5 \in\mathbb{C}^n$ and $w_3\in\mathbb{C}$.
It follows that
\begin{equation} 
\label{eq:factor1:2n+1-odd}
\begin{split}
A & \sim\frac1{\sqrt2}
       \begin{pmatrix}
        I_n+W_1        & \sqrt2\mathbf{w}_1 & (-1)^nC_n^*+(-1)^{n+1}W_1C_n^* \\
        \mathbf{w}_2^T & \sqrt2 w_3         & (-1)^{n+1}\mathbf{w}_2^TC_n^* \\
        W_3            & \sqrt2\mathbf{w}_5 & (-1)^{n+1}W_3C_n^*
       \end{pmatrix} \\
  &   = \frac1{\sqrt2}
        \begin{pmatrix} W_1 & \mathbf{w}_1 & I_n \\ \mathbf{w}_2^T & w_3 & 0 \\ W_3 & \mathbf{w}_5 & 0 \end{pmatrix}
        \begin{pmatrix}
         I_n &      0 & (-1)^{n+1}C_n^* \\
           0 & \sqrt2 &             0 \\
         I_n &        0 &   (-1)^{n}C_n^* \\
        \end{pmatrix}, \\
B & \sim \frac1{\sqrt2}
       \begin{pmatrix}
                   W_2 & 0      & (-1)^{n}W_2C_n^* \\
        \mathbf{w}_4^T & \sqrt2 & (-1)^{n}\mathbf{w}_4^TC_n^*  \\
               I_n+W_4 & 0      & (-1)^{n+1}C_n^*+(-1)^{n}W_4C_n^*
       \end{pmatrix} \\
  &   = \frac1{\sqrt2}
        \begin{pmatrix} 0 & 0 & W_2 \\ 0 & 1 & \mathbf{w}_4^T \\ I_n & 0 & W_4 \end{pmatrix}
        \begin{pmatrix}
         I_n &      0 & (-1)^{n+1}C_n^* \\
           0 & \sqrt2 &             0 \\
         I_n &      0 &   (-1)^{n}C_n^* \\
        \end{pmatrix},
\end{split}
\end{equation}
and hence,
\begin{align}
 \rank(A) &= n+\rank\begin{pmatrix}\mathbf{w}_2^T & w_3 \\ W_3 & \mathbf{w}_5 \end{pmatrix}, \label{eq:rankA:2n+1-odd} \\
 \rank(B) &= n+1+\rank(W_2). \label{eq:rankB:2n+1-odd}
\end{align}
Again we use the CS-decomposition. %described in detail in \cite{paige94} and \cite[Theorem 2.7.1]{horn12},
%provides a useful way to speak about rank.
In particular, \cite[Theorem 4.1]{fuhr18}\footnote{using the embedding $U\mapsto\begin{pmatrix}U&0\\0&1\end{pmatrix}$} yields
\begin{equation*}
 W = \begin{pmatrix} G_1 & 0 \\ 0 & G_2 \end{pmatrix}
     \begin{pmatrix} (I_n+F)/2 & 0 & (I_n-F)/2 \\ 0 & 1 & 0 \\ (I_n-F)/2 & 0 & (I_n+F)/2 \end{pmatrix}
     \begin{pmatrix} G_3 & 0 \\ 0 & G_4 \end{pmatrix}
\end{equation*}
for some $(n+1)\times(n+1)$ unitary matrices $G_1$, $G_3$, and $n\times n$ unitary matrices $F$, $G_2$ and $G_4$.
Since
\begin{equation*}
 \begin{pmatrix} (I_n+F)/2 & (I_n-F)/2 \\ (I_n-F)/2 & (I_n+F)/2 \end{pmatrix}
\end{equation*}
is a $2n\times 2n$ unitary matrix, $W$ may be expressed in the following form, using \cite[Corollary 3.1]{fuhr18},
\begin{equation*}
 W = \begin{pmatrix} U_1 & 0 \\ 0 & U_2 \end{pmatrix}
     \begin{pmatrix} C & 0 & S \\ 0 & 1 & 0 \\ -S & 0 & C \end{pmatrix}
     \begin{pmatrix} V_1 & 0 \\ 0 & V_2 \end{pmatrix}
\end{equation*}
where $U_1$ and $V_1$ are $(n+1)\times(n+1)$ unitary matrices, $U_2$ and $V_2$ are $n\times n$ unitary matrices,
and $C$ and $S$ are positive semi-definite diagonal matrices satisfying $C^2+S^2=I_n$.
Consequently,
\begin{equation*}
 W = \begin{pmatrix}
        U_1\begin{bmatrix} C & 0 \\ 0 &  1\end{bmatrix}V_1 & U_1\begin{bmatrix}S \\ 0\end{bmatrix}V_2 \\[8pt]
       -U_2\begin{bmatrix} S & 0\end{bmatrix}V_1 & U_2CV_2
      \end{pmatrix}
\end{equation*}
so that, using $\mathbf{S}=\begin{bmatrix}S \\ 0\end{bmatrix}$ and $\mathbf{C}=C\oplus 1=\begin{bmatrix}C&0\\0&1\end{bmatrix}$,
\begin{equation}
 \label{eq:W:2n+1-odd}
 \begin{pmatrix} W_1 & \mathbf{w}_1 & W_2 \\ \mathbf{w}_2^T & w_3 & \mathbf{w}_4^T \\ W_3 & \mathbf{w}_5 & W_4 \end{pmatrix}
 =
 \begin{pmatrix}
  U_1\mathbf{C}V_1 & U_1\mathbf{S}V_2 \\
  -U_2\mathbf{S}^*V_1 & U_2CV_2
 \end{pmatrix}
\end{equation}
It follows from \eqref{eq:factor1:2n+1-odd} that
\begin{equation*}
 (A:B) \sim
 \dfrac{1}{\sqrt{2}}
 \begin{pmatrix}
  W_1 & \mathbf{w}_1 & I_n
  & 0_n & 0 & W_2 \\
  \mathbf{w}_2^T & w_3 & 0
  & 0 & 1 & \mathbf{w}_4^T \\
  W_3 & \mathbf{w}_5 & 0_n
  & I_n & 0 & W_4
 \end{pmatrix}Q_4
\end{equation*}
where
\begin{equation*}
 Q_4 = 
 \begin{pmatrix}
  I_n &      0    & (-1)^{n+1}C_n^* & 0 & 0 & 0 \\
    0 & \sqrt2    &               0 & 0 & 0 & 0 \\
  I_n &      0    &   (-1)^{n}C_n^* & 0 & 0 & 0 \\
  0 & 0 & 0 & I_n &               0 &(-1)^{n+1}C_n^* \\
  0 & 0 & 0 &   0 &          \sqrt2 &             0 \\
  0 & 0 & 0 & I_n &               0 &  (-1)^{n}C_n^*
 \end{pmatrix}.
\end{equation*}
The decomposition of $(A:B)$ now follows from
\begin{align*}
 (A:B) & \sim
 \dfrac{1}{\sqrt{2}}
 \begin{pmatrix}
  U_1\mathbf{C}V_1 & \begin{bmatrix} I_n & 0_n & 0\\ 0 & 0 & 1\end{bmatrix} & U_1\mathbf{S}V_2 \\[10pt]
  -U_2\mathbf{S}^*V_1 & \begin{bmatrix} 0_n & I_n & 0 \end{bmatrix} & U_2CV_2
 \end{pmatrix}Q_4 \\
 &=
 \dfrac{1}{\sqrt{2}}
 \begin{pmatrix}
  U_1 & 0 \\
    0 & U_2
 \end{pmatrix}
 \begin{pmatrix}
  \mathbf{C} & U_1^*\begin{bmatrix} I_n & 0_n & 0\\ 0 & 0 & 1\end{bmatrix} & \mathbf{S} \\[10pt]
  -\mathbf{S}^* & \begin{bmatrix} 0_n & U_2^* & 0 \end{bmatrix} & C
 \end{pmatrix}
 \begin{pmatrix}
  V_1 & 0 & 0 \\
    0 & I_{2n+1} & 0  \\
  0 & 0 & V_2
 \end{pmatrix}Q_4 \\
 &=
 \dfrac{1}{\sqrt{2}}
 \begin{pmatrix}
  U_1 & 0 \\
    0 & U_2
 \end{pmatrix}
 \begin{pmatrix}
  \mathbf{C} & I_{n+1} & 0 & I_{n+1} & \mathbf{S} \\
  -\mathbf{S}^* & 0 & I_n & 0 & C
 \end{pmatrix} \\
 & \qquad \times
 \begin{pmatrix}
  V_1 & 0 & 0 & 0 & 0 \\
    0 & U_1^*\begin{bmatrix}I_n\\ 0 \end{bmatrix} & 0 & 0 & 0 \\
    0 & 0 & U_2^* & 0 & 0  \\
    0 & 0 & 0 & U_1^*\begin{bmatrix}0\\\vdots\\0\\1\end{bmatrix} & 0 \\
  0 & 0 & 0 & 0 & V_2
 \end{pmatrix} Q_4 \\
 &=
 \dfrac{1}{\sqrt{2}}
 \begin{pmatrix}
  U_1 & 0 \\
    0 & U_2
 \end{pmatrix}
 \begin{pmatrix}
  \mathbf{C} & I_{n+1} & 0 & I_{n+1} & \mathbf{S} \\
  -\mathbf{S}^* & 0 & I_n & 0 & C
 \end{pmatrix}
 \begin{pmatrix}
  V_1 & 0 & 0 & 0 & 0 \\
    0 & U_1^* & 0 & 0 & 0 \\
    0 & 0 & U_2^* & 0 & 0 \\
    0 & 0 & 0 & U_1^* & 0  \\
  0 & 0 & 0 & 0 & V_2
 \end{pmatrix}
 \\
 &\qquad \times
 \begin{pmatrix}
  I_{n+1} & 0 & 0 & 0 & 0 \\
    0 & \begin{bmatrix}I_n \\ 0 \end{bmatrix} & 0 & 0 & 0 \\
    0 & 0 & I_n & 0 & 0 \\
    0 & 0 & 0 & \begin{bmatrix}0\\\vdots \\0\\1\end{bmatrix} & 0  \\
  0 & 0 & 0 & 0 & I_n
 \end{pmatrix} Q_4.
\end{align*}
We note that, by \eqref{eq:W:2n+1-odd},
\begin{align*}
 \begin{pmatrix}\mathbf{w}_2^T & w_3 \\ W_3 & \mathbf{w}_5 \end{pmatrix}
 &=
 \begin{pmatrix}0 & I_{n+1}\end{pmatrix}
 \begin{pmatrix}
  U_1\mathbf{C}V_1 & U_1\mathbf{S}V_2 \\
  -U_2\mathbf{S}^*V_1 & U_2CV_2
 \end{pmatrix}
 \begin{pmatrix}I_{n+1} \\ 0\end{pmatrix} \\
 &=
 \begin{pmatrix}\begin{bmatrix}\mathbf{0}_n^T & 1\end{bmatrix} & \mathbf{0}_n^T \\ 0 & I_{n}\end{pmatrix}
 \begin{pmatrix}
  U_1\mathbf{C}V_1 & U_1\mathbf{S}V_2 \\
  -U_2\mathbf{S}^*V_1 & U_2CV_2
 \end{pmatrix}
 \begin{pmatrix}I_{n+1} \\ 0\end{pmatrix} \\
 &= \begin{pmatrix}
     \begin{bmatrix}\mathbf{0}_n^T&1\end{bmatrix}U_1\mathbf{C}V_1 \\
     -U_2\mathbf{S}^*V_1
    \end{pmatrix},\\
 W_2
 &=
 \begin{pmatrix}\begin{bmatrix}I_n & 0\end{bmatrix} & 0\end{pmatrix}
 \begin{pmatrix}
  U_1\mathbf{C}V_1 & U_1\mathbf{S}V_2 \\
  -U_2\mathbf{S}^*V_1 & U_2CV_2
 \end{pmatrix}
 \begin{pmatrix}0 \\ I_{n}\end{pmatrix} \\
 &= \left(\begin{bmatrix} I_n & 0\end{bmatrix} U_1\mathbf{S}V_2\right),
\end{align*}
where $0$ denotes a zero matrix of appropriate size in each case.
Hence, it follows from \eqref{eq:rankA:2n+1-odd}, \eqref{eq:rankB:2n+1-odd} and \eqref{eq:W:2n+1-odd} that
\begin{align*}
 \rank(A)&=n+\rank\begin{pmatrix}
               \begin{bmatrix}\mathbf{0}_n^T&1\end{bmatrix}U_1\mathbf{C}V_1 \\
               -U_2\mathbf{S}^*V_1
             \end{pmatrix}
          =n+\rank\begin{pmatrix}
               \begin{bmatrix}\mathbf{0}_n^T&1\end{bmatrix}U_1\mathbf{C} \\
               \mathbf{S}^*
                  \end{pmatrix}, \\
 \rank(B)&=n+1+\rank\left(\begin{bmatrix} I_n & 0\end{bmatrix}
                         U_1\mathbf{S}V_2\right)
          =n+1+\rank\left(\begin{bmatrix} I_n & 0 \end{bmatrix}U_1\mathbf{S}\right),
\end{align*}
by invariance of the rank under invertible column operations (in this case: multiplication by $V_1$ and $V_2$ respectively).
By construction $\mathbf{S}\mathbf{S}^*=I_{n+1}-\mathbf{C}\mathbf{C}^*=I_{n+1}-\mathbf{C}^2$,
and since $\rank(X)=\rank(XX^*)=\rank(X^*X)$ for any matrix $X$, we have
\begin{align*}
  \left(\begin{bmatrix} I_n & 0 \end{bmatrix}U_1\mathbf{S}\right)
  \left(\begin{bmatrix} I_n & 0 \end{bmatrix}U_1\mathbf{S}\right)^*
 &= \begin{pmatrix} I_n & 0 \end{pmatrix}
    U_1\mathbf{S}\mathbf{S}^*U_1^*
    \begin{pmatrix} I_n \\ 0 \end{pmatrix}\\
 &= I_n
  - \begin{pmatrix} I_n & 0 \end{pmatrix}
    U_1\mathbf{C}^2U_1^*
    \begin{pmatrix} I_n \\ 0 \end{pmatrix} \\
 \begin{pmatrix} \begin{bmatrix}\mathbf{0}_n^T & 1\end{bmatrix}U_1\mathbf{C} \\ \mathbf{S}^* \end{pmatrix}^*
 \begin{pmatrix} \begin{bmatrix}\mathbf{0}_n^T & 1\end{bmatrix}U_1\mathbf{C} \\ \mathbf{S}^* \end{pmatrix}
  &= \mathbf{S}\mathbf{S}^*
   + \mathbf{C}^*U_1^*
     \begin{pmatrix}0_n&0\\0&1\end{pmatrix}
     U_1\mathbf{C} \\
  &= I_{n+1}-\mathbf{C}^*U_1^*
     \begin{pmatrix}I_n&0\\0&0\end{pmatrix}
     U_1\mathbf{C} \\
  &= I_{n+1}-\mathbf{C}U_1^*
     \begin{pmatrix} I_n \\ 0 \end{pmatrix}
     \begin{pmatrix} I_n & 0 \end{pmatrix}
     U_1\mathbf{C}.
\end{align*}
Hence,
\begin{align*}
 \rank(A)&=n+
  \rank\left(I_{n+1} - K^*K\right), \\
 \rank(B)&=n+1+
  \rank\left(I_n - KK^*\right),
\end{align*}
where
\begin{equation*}
 K=\begin{pmatrix}I_n&0\end{pmatrix}U_1\mathbf{C}.
\end{equation*}
Since $K^*K$ and $KK^*$ have the same non-zero eigenvalues with the
same multiplicities $KK^*$ and $K^*K$ have the same number of
unit eigenvalues (i.e. equal to 1) so $\Null(I_n-KK^*)=\Null(I_{n+1}-K^*K)$.
By the rank-nullity duality
\begin{equation*}
 \rank(A) = 2n+1-\Null(I_{n+1}-K^*K) = 2n+1-\Null(I_n-KK^*) = \rank(B).
 \qedhere
\end{equation*}
\end{proof}

\begin{proof}[Proof of Theorem \ref{thm:2n+1th-even}]
Following the proof of Theorem \ref{thm:2n+1th-odd}, we
now express $V$ which obeys
\begin{equation}
 \label{eq:decompigen-even}
 \begin{pmatrix}C_{2n+1} & 0 \\ 0 & -C_{2n+1}\end{pmatrix}
 =
 V\begin{pmatrix}-I_{2n+1} & 0 \\ 0 & I_{2n+1}\end{pmatrix}V^*.
\end{equation}
In particular, we will choose
\begin{equation}
 \label{eq:Vsimp2n+1-even}
 V =
 \begin{pmatrix}
   V_{-1} & \mathbf{v}_{-1} & 0  & V_1 & 0               & 0 \\
   0      & 0               & V_1& 0   & \mathbf{v}_{-1} & V_{-1}
 \end{pmatrix},
\end{equation}
where $V_1$ and $V_{-1}$ are given in \eqref{eq:V1} and \eqref{eq:V-1} and
\begin{equation*}
 \mathbf{v}_{-1} =
  \underbrace{
   \begin{pmatrix}
   0 &
   \cdots &
   0 &
   1 &
   0 &
   \cdots &
   0
   \end{pmatrix}^T
   \!\!\!
  }_{\text{$(n+1)$-th coordinate}}
\end{equation*}
so that
\begin{align*}
% \label{eq:vblock2n+1-even}
  V_{11} &= \frac1{\sqrt2}\begin{pmatrix} I_n & 0 & 0_n \\ 0 & \sqrt2 & 0 \\ (-1)^{n}C_n & 0 & 0_n \end{pmatrix},\\
  V_{12} &= \frac1{\sqrt2}\begin{pmatrix} I_n & 0 & 0_n \\ 0 & 0      & 0 \\ (-1)^{n+1}C_n & 0 & 0_n \end{pmatrix},\\
  V_{21} &= \frac1{\sqrt2}\begin{pmatrix} 0_n & 0 & I_n \\ 0 & 0      & 0 \\ 0_n & 0 & (-1)^{n+1}C_n \end{pmatrix},\\
  V_{22} &= \frac1{\sqrt2}\begin{pmatrix} 0_n & 0 & I_n \\ 0 & \sqrt2 & 0 \\ 0_n & 0 & (-1)^nC_n \end{pmatrix}.
\end{align*}
Choosing appropriate $W$ provides the remaining canonical forms. Thus
\begin{align*} 
A & \sim \frac1{\sqrt2}\left[
       \begin{pmatrix}
        I_n & 0            & (-1)^nC_n^* \\
          0 & \sqrt2       & 0  \\
        0_n & 0            & 0_n
       \end{pmatrix}
       +
       W
       \begin{pmatrix}
        I_n & 0            & (-1)^{n+1}C_n^* \\
          0 & 0            & 0  \\
        0_n & 0            & 0_n
       \end{pmatrix}
      \right], \\
B & \sim \frac1{\sqrt2}\left[
       \begin{pmatrix}
        0_n & 0            & 0_n \\
          0 & 0            & 0 \\
        I_n & 0            & (-1)^{n+1}C_n^*
       \end{pmatrix}
       +
       W
       \begin{pmatrix}
        0_n & 0            & 0_n \\
          0 & \sqrt2       & 0 \\
        I_n & 0            & (-1)^nC_n^*
       \end{pmatrix}
      \right].
\end{align*}
Let
\begin{equation*}
 W = \begin{pmatrix} W_1 & \mathbf{w}_1 & W_2 \\ \mathbf{w}_2^T & w_3 & \mathbf{w}_4^T \\ W_3 & \mathbf{w}_5 & W_4 \end{pmatrix}
\end{equation*}
where $W_1, W_2,W_3,W_4 \in M_n(\mathbb{C})$, $\mathbf{w}_1, \mathbf{w}_2, \mathbf{w}_4, \mathbf{w}_5 \in\mathbb{C}^n$ and $w_3\in\mathbb{C}$.
It follows that
\begin{equation} 
\label{eq:factor1:2n+1-even}
\begin{split}
A & \sim\frac1{\sqrt2}
       \begin{pmatrix}
        I_n+W_1        & 0      & (-1)^nC_n^*+(-1)^{n+1}W_1C_n^* \\
        \mathbf{w}_2^T & \sqrt2 & (-1)^{n+1}\mathbf{w}_2^TC_n^* \\
        W_3            & 0      & (-1)^{n+1}W_3C_n^*
       \end{pmatrix} \\
  &   = \frac1{\sqrt2}
        \begin{pmatrix} W_1 & 0 & I_n \\ \mathbf{w}_2^T & 1 & 0 \\ W_3 & 0 & 0 \end{pmatrix}
        \begin{pmatrix}
         I_n &      0 & (-1)^{n+1}C_n^* \\
           0 & \sqrt2 &             0 \\
         I_n &      0 &   (-1)^{n}C_n^* \\
        \end{pmatrix}, \\
B & \sim\frac1{\sqrt2}
       \begin{pmatrix}
                   W_2 & \sqrt2\mathbf{w}_1 & (-1)^{n}W_2C_n^* \\
        \mathbf{w}_4^T & \sqrt2 w_3         & (-1)^{n}\mathbf{w}_4^TC_n^*  \\
               I_n+W_4 & \sqrt2\mathbf{w}_5 & (-1)^{n+1}C_n^*+(-1)^{n}W_4C_n^*
       \end{pmatrix} \\
  &   = \frac1{\sqrt2}
        \begin{pmatrix} 0 & \mathbf{w}_1 & W_2 \\ 0 & w_3 & \mathbf{w}_4^T \\ I_n & \mathbf{w}_5 & W_4 \end{pmatrix}
        \begin{pmatrix}
         I_n &      0 & (-1)^{n+1}C_n^* \\
           0 & \sqrt2 &             0 \\
         I_n &      0 &   (-1)^{n}C_n^* \\
        \end{pmatrix},
\end{split}
\end{equation}
and hence,
\begin{align}
 \rank(A) &= n+1+\rank(W_3), \label{eq:rankA:2n+1-even} \\
 \rank(B) &= n+\rank\begin{pmatrix}\mathbf{w}_1 & W_2 \\ w_3 & \mathbf{w}_4^T \end{pmatrix}, \label{eq:rankB:2n+1-even}
\end{align}
Again we use the CS-decomposition. %described in detail in \cite{paige94} and \cite[Theorem 2.7.1]{horn12},
%provides a useful way to speak about rank.
In particular, \cite[Theorem 4.1]{fuhr18} yields
\begin{equation*}
 W = \begin{pmatrix} G_1 & 0 \\ 0 & G_2 \end{pmatrix}
     \begin{pmatrix} (I_n+F)/2 & 0 & (I_n-F)/2 \\ 0 & 1 & 0 \\ (I_n-F)/2 & 0 & (I_n+F)/2 \end{pmatrix}
     \begin{pmatrix} G_3 & 0 \\ 0 & G_4 \end{pmatrix}
\end{equation*}
for some $n\times n$ unitary matrices $F$, $G_1$, $G_3$, and $(n+1)\times(n+1)$ unitary matrices $G_2$ and $G_4$.
Since
\begin{equation*}
 \begin{pmatrix} (I_n+F)/2 & (I_n-F)/2 \\ (I_n-F)/2 & (I_n+F)/2 \end{pmatrix}
\end{equation*}
is a $2n\times 2n$ unitary matrix, $W$ may be expressed in the following form, using \cite[Corollary 3.1]{fuhr18},
\begin{equation*}
 W = \begin{pmatrix} U_1 & 0 \\ 0 & U_2 \end{pmatrix}
     \begin{pmatrix} C & 0 & S \\ 0 & 1 & 0 \\ -S & 0 & C \end{pmatrix}
     \begin{pmatrix} V_1 & 0 \\ 0 & V_2 \end{pmatrix}
\end{equation*}
where $U_1$ and $V_1$ are $n\times n$ unitary matrices, $U_2$ and $V_2$ are $(n+1)\times (n+1)$ unitary matrices,
and $C$ and $S$ are positive semi-definite diagonal matrices satisfying $C^2+S^2=I_n$.
Consequently,
\begin{equation*}
 W = \begin{pmatrix}
       U_1CV_1 & U_1\begin{bmatrix}0 & S\end{bmatrix}V_2 \\
       U_2\begin{bmatrix} 0 \\ -S\end{bmatrix}V_1 & U_2\begin{bmatrix} 1 & 0 \\ 0 &  C\end{bmatrix}V_2
      \end{pmatrix}
\end{equation*}
so that, using $\mathbf{S}=\begin{bmatrix}0&S\end{bmatrix}$ and $\mathbf{C}=1\oplus C=\begin{bmatrix}1&0\\0&C\end{bmatrix}$,
\begin{equation}
 \label{eq:W:2n+1-even}
 \begin{pmatrix} W_1 & \mathbf{w}_1 & W_2 \\ \mathbf{w}_2^T & w_3 & \mathbf{w}_4^T \\ W_3 & \mathbf{w}_5 & W_4 \end{pmatrix}
 =
 \begin{pmatrix}
  U_1CV_1 & U_1\mathbf{S}V_2 \\
  -U_2\mathbf{S}^*V_1 & U_2\mathbf{C}V_2
 \end{pmatrix}
\end{equation}
Hence, \eqref{eq:factor1:2n+1-even} gives
\begin{equation*}
 (A:B) \sim
 \dfrac{1}{\sqrt{2}}
 \begin{pmatrix}
  W_1 & 0 & I_n
  & 0 & \mathbf{w}_1 & W_2 \\
  \mathbf{w}_2^T & 1 & 0
  & 0 & w_3 & \mathbf{w}_4^T \\
  W_3 & 0 & 0
  & I_n & \mathbf{w}_5 & W_4
 \end{pmatrix}Q_4
\end{equation*}
where
\begin{equation*}
 Q_4 = 
 \begin{pmatrix}
  I_n &      0    & (-1)^{n+1}C_n^* & 0 & 0 & 0 \\
    0 & \sqrt2    &               0 & 0 & 0 & 0 \\
  I_n &      0    &   (-1)^{n}C_n^* & 0 & 0 & 0 \\
  0 & 0 & 0 & I_n &               0 &(-1)^{n+1}C_n^* \\
  0 & 0 & 0 &   0 &          \sqrt2 &             0 \\
  0 & 0 & 0 & I_n &               0 &  (-1)^{n}C_n^*
 \end{pmatrix}.
\end{equation*}
It follows that
\begin{align*}
 (A:B) & \sim
 \dfrac{1}{\sqrt{2}}
 \begin{pmatrix}
  U_1CV_1 & \begin{bmatrix} 0 & I_n & 0_n\end{bmatrix} & U_1\mathbf{S}V_2 \\[4pt]
  -U_2\mathbf{S}^*V_1 & \begin{bmatrix} 1 & 0 & 0 \\ 0 & 0_n & I_n \end{bmatrix} & U_2\mathbf{C}V_2
 \end{pmatrix}Q_4 \\
 &=
 \dfrac{1}{\sqrt{2}}
 \begin{pmatrix}
  U_1 & 0 \\
    0 & U_2
 \end{pmatrix}
 \begin{pmatrix}
  C & \begin{bmatrix} 0 & U_1^* & 0_n\end{bmatrix} & \mathbf{S} \\[4pt]
  -\mathbf{S}^* & U_2^*\begin{bmatrix} 1 & 0 & 0 \\ 0 & 0_n & I_n \end{bmatrix} & \mathbf{C}
 \end{pmatrix}
 \begin{pmatrix}
  V_1 & 0 & 0 \\
    0 & I_{2n+1} & 0  \\
  0 & 0 & V_2
 \end{pmatrix}Q_4 \\
 &=
 \dfrac{1}{\sqrt{2}}
 \begin{pmatrix}
  U_1 & 0 \\
    0 & U_2
 \end{pmatrix}
 \begin{pmatrix}
  C & 0 & I_n & 0 & \mathbf{S} \\
  -\mathbf{S}^* & I_{n+1} & 0 & I_{n+1} & \mathbf{C}
 \end{pmatrix} \\
 & \qquad \times
 \begin{pmatrix}
  V_1 & 0 & 0 & 0 & 0 \\
    0 & U_2^*\begin{bmatrix}1\\0\\\vdots\\0\end{bmatrix} & 0 & 0 & 0 \\
    0 & 0 & U_1^* & 0 & 0 \\
    0 & 0 & 0 & U_2^*\begin{bmatrix}0\\ I_n \end{bmatrix} & 0  \\
  0 & 0 & 0 & 0 & V_2
 \end{pmatrix} Q_4 \\
 &=
 \dfrac{1}{\sqrt{2}}
 \begin{pmatrix}
  U_1 & 0 \\
    0 & U_2
 \end{pmatrix}
 \begin{pmatrix}
  C & 0 & I_n & 0 & \mathbf{S} \\
  -\mathbf{S}^* & I_{n+1} & 0 & I_{n+1} & \mathbf{C}
 \end{pmatrix}
 \begin{pmatrix}
  V_1 & 0 & 0 & 0 & 0 \\
    0 & U_2^* & 0 & 0 & 0 \\
    0 & 0 & U_1^* & 0 & 0 \\
    0 & 0 & 0 & U_2^* & 0  \\
  0 & 0 & 0 & 0 & V_2
 \end{pmatrix}
 \\
 &\qquad \times
 \begin{pmatrix}
  I_n & 0 & 0 & 0 & 0 \\
    0 & \begin{bmatrix}1\\0\\ \vdots \\0\end{bmatrix} & 0 & 0 & 0 \\
    0 & 0 & I_n & 0 & 0 \\
    0 & 0 & 0 & \begin{bmatrix}0\\ I_n \end{bmatrix} & 0  \\
  0 & 0 & 0 & 0 & I_{n+1}
 \end{pmatrix} Q_4.
\end{align*}
We note that, by \eqref{eq:W:2n+1-even},
\begin{align*}
 \begin{pmatrix} \mathbf{w}_1 & W_2 \\  w_3 & \mathbf{w}_4^T \end{pmatrix}
 &=
 \begin{pmatrix} I_{n+1} & 0 \end{pmatrix}
 \begin{pmatrix}
  U_1\mathbf{C}V_1 & U_1\mathbf{S}V_2 \\
  -U_2\mathbf{S}^*V_1 & U_2CV_2
 \end{pmatrix}
 \begin{pmatrix} 0 \\ I_{n+1} \end{pmatrix} \\
 &=
 \begin{pmatrix}I_n & 0 \\ \mathbf{0}_n^T & \begin{bmatrix}1 & \mathbf{0}_n^T\end{bmatrix}\end{pmatrix}
 \begin{pmatrix}
  U_1\mathbf{C}V_1 & U_1\mathbf{S}V_2 \\
  -U_2\mathbf{S}^*V_1 & U_2CV_2
 \end{pmatrix}
 \begin{pmatrix} 0 \\ I_{n+1} \end{pmatrix} \\
 &= \begin{pmatrix}
     U_1\mathbf{S}V_2 \\
     \begin{bmatrix}1&\mathbf{0}_n^T\end{bmatrix}U_2\mathbf{C}V_2
    \end{pmatrix}, \\
 W_3
 &=
 \begin{pmatrix}0 & \begin{bmatrix}0 & I_n\end{bmatrix}\end{pmatrix}
 \begin{pmatrix}
  U_1\mathbf{C}V_1 & U_1\mathbf{S}V_2 \\
  -U_2\mathbf{S}^*V_1 & U_2CV_2
 \end{pmatrix}
 \begin{pmatrix}I_{n} \\ 0\end{pmatrix} \\
 &= \left(\begin{bmatrix} 0 & I_n \end{bmatrix} U_2(-\mathbf{S}^*)V_1\right).
\end{align*}
It follows from \eqref{eq:rankA:2n+1-even}, \eqref{eq:rankB:2n+1-even} and \eqref{eq:W:2n+1-even} that
\begin{align*}
 \rank(A)&=n+1+\rank\left(\begin{bmatrix} 0 & I_n \end{bmatrix}
                         U_2(-\mathbf{S}^*)V_1\right)
          =n+1+\rank\left(\begin{bmatrix} 0 & I_n \end{bmatrix}U_2\mathbf{S}^*\right), \\
 \rank(B)&=n+\rank\begin{pmatrix}
               U_1\mathbf{S}V_2 \\
               \begin{bmatrix}1&\mathbf{0}_n^T\end{bmatrix}U_2\mathbf{C}V_2
             \end{pmatrix}
          =n+\rank\begin{pmatrix}
                   \mathbf{S} \\
                   \begin{bmatrix}1&\mathbf{0}_n^T\end{bmatrix}U_2\mathbf{C}
                  \end{pmatrix}.
\end{align*}
By construction $\mathbf{S}^*\mathbf{S}=I_{n+1}-\mathbf{C}^*\mathbf{C}=I_{n+1}-\mathbf{C}^2$,
and since $\rank(X)=\rank(XX^*)=\rank(X^*X)$ for any matrix $X$, we have
\begin{align*}
  \left(\begin{bmatrix} 0 & I_n \end{bmatrix}U_2\mathbf{S}^*\right)
  \left(\begin{bmatrix} 0 & I_n \end{bmatrix}U_2\mathbf{S}^*\right)^*
 &= \begin{pmatrix} 0 & I_n \end{pmatrix}
    U_2\mathbf{S}^*\mathbf{S}U_2^*
    \begin{pmatrix} 0 \\ I_n \end{pmatrix}\\
 &= I_n
  - \begin{pmatrix} 0 & I_n \end{pmatrix}
    U_2\mathbf{C}^2U_2^*
    \begin{pmatrix} 0 \\ I_n \end{pmatrix} \\
 \begin{pmatrix} \mathbf{S} \\ \begin{bmatrix}1&\mathbf{0}_n^T\end{bmatrix}U_2\mathbf{C} \end{pmatrix}^*
 \begin{pmatrix} \mathbf{S} \\ \begin{bmatrix}1&\mathbf{0}_n^T\end{bmatrix}U_2\mathbf{C} \end{pmatrix}
  &= \mathbf{S}^*\mathbf{S}
   + \mathbf{C}^*U_2^*
     \begin{pmatrix}1&0\\0&0_n\end{pmatrix}
     U_2\mathbf{C} \\
  &= I_{n+1}-\mathbf{C}^*U_2^*
     \begin{pmatrix}0&0\\0&I_n\end{pmatrix}
     U_2\mathbf{C} \\
  &= I_{n+1}-\mathbf{C}U_2^*
     \begin{pmatrix} 0 \\ I_n \end{pmatrix}
     \begin{pmatrix} 0 & I_n \end{pmatrix}
     U_2\mathbf{C}.
\end{align*}
Hence,
\begin{align*}
 \rank(A)&=n+1+
  \rank\left(I_n - KK^*\right), \\
 \rank(B)&=n+
  \rank\left(I_{n+1} - K^*K\right)
\end{align*}
where
\begin{equation*}
 K=\begin{pmatrix}0&I_n\end{pmatrix}U_2\mathbf{C}.
\end{equation*}
Since $K^*K$ and $KK^*$ have the same non-zero eigenvalues with the
same multiplicities $KK^*$ and $K^*K$ have the same number of
unit eigenvalues (i.e. equal to 1) so $\Null(I_n-KK^*)=\Null(I_{n+1}-K^*K)$.
By the rank-nullity duality
\begin{equation*}
 \rank(A) = 2n+1-\Null(I_n-KK^*) = 2n+1-\Null(I_{n+1}-K^*K) = \rank(B).
 \qedhere
\end{equation*}
\end{proof}

\section{Conclusion}

This paper provides the remaining canonical forms for odd-order differential operators with self-adjoint boundary conditions.
To summarize, let us present all of the canonical forms:

\begin{theorem*}[Canonical forms for $2n$-th order differential operators {\cite[Theorem 5.1]{hardy24}}]
 \strut\\
 Let $A$ and $B$ be $2n\times 2n$ matrices satisfying
 \begin{equation*}
  \rank(A:B)=2n \quad \textrm{and}\quad AC_{2n}A^*=BC_{2n}B^*.
 \end{equation*}
 Let $Z$ be the matrix
 \begin{equation*}
  Z = \frac1{\sqrt2}
      \begin{pmatrix}
       I_n & I_n & 0   & 0   \\
       I_n &-I_n & 0   & 0   \\
       0 & 0   & I_n & I_n \\
       0 & 0   & I_n &-I_n
      \end{pmatrix}
      \begin{pmatrix}
       I_n & 0   & 0   & 0 \\
       0 &(-1)^{n+1}iC_n & 0   & 0 \\
       0 & 0   & I_n & 0 \\
       0 & 0   & 0   & (-1)^{n+1}iC_n
      \end{pmatrix}.
 \end{equation*}
 Then there exists a $2n\times 2n$ non singular matrix $U$ and
 $n\times n$ unitary matrices $V_1$, $U_1^*$, $U_2^*$ and $V_2$, and
 positive semi-definite diagonal matrices $C$ and $S$ with $C^2+S^2=I_n$,
 such that
 \begin{equation*}
  (A:B) \sim U
          \begin{pmatrix}
            C & I_n & 0 & S \\
           -S & 0 & I_n & C
          \end{pmatrix}
          \begin{pmatrix}
           V_1 & 0   & 0   & 0 \\
             0 & U_1^* & 0   & 0 \\
             0 & 0   & U_2^* & 0 \\
             0 & 0   & 0   & V_2
          \end{pmatrix}
          Z,
 \end{equation*}
 where $(A:B) \sim (A':B')$ if and only if there exists $J$ such that $J^{-1}(A:B) = (A':B')$.
 The boundary conditions are
 \begin{enumerate}
  \item separated, if and only if $S=0$.
  \item mixed, if and only $0<\rank(S)<n$.
  \item coupled, if and only if $\rank(S)=n$.
 \end{enumerate}
\end{theorem*}

\begin{theorem*}[Canonical forms for $2(2k)+1$-th order differential operators]
 \strut\\
 Let $2k=n\in\mathbb{N}$ be even and let $A$ and $B$ be $(2n+1)\times (2n+1)$ matrices satisfying 
 \begin{equation*}
  \rank(A:B)=2n+1 \quad \textrm{and}\quad AC_{2n+1}A^*=BC_{2n+1}B^*.
 \end{equation*}
 Then, there exist $(n+1)\times(n+1)$ unitary matrices $U_1$ and $V_1$,
 $n\times n$ unitary matrices $U_2$ and $V_2$
 and real diagonal $n\times n$ matrices $C$ and $S$
 such that
 \begin{enumerate}[label=(\alph*)]
  \item
   $C^2+S^2=I_n$,
  \item
   $(A:B)$ has the canonical form
   \begin{equation*}
   (A:B) \sim \dfrac{1}{\sqrt{2}} Q_1
           \begin{pmatrix}
            C & 0 & I_n & 0 & \mathbf{S} \\
            -\mathbf{S}^* & I_{n+1} & 0 & I_{n+1} & \mathbf{C}
           \end{pmatrix}
           Q_2,
   \end{equation*}
   where $(A:B) \sim (A':B')$ if and only if there exists $J$ such that $J^{-1}(A:B) = (A':B')$,
  \item $\rank(A) = 2n+1-\Null(I_n-KK^*) = \rank(B)$,
  \item the boundary conditions are
   \begin{enumerate}[label=(\roman*)]
    \item mixed, if and only if $\rank(I_n-KK^*)<n$,
    \item coupled, if and only if $\rank(I_n-KK^*)=n$,
   \end{enumerate}
 \end{enumerate}
 \medskip
 where
 $\mathbf{S}=\begin{pmatrix} 0 & S \end{pmatrix}$,
 $\mathbf{C}=1\oplus C=\begin{pmatrix}1&0\\0&C\end{pmatrix}$,
 $K=\begin{pmatrix}0&I_n\end{pmatrix}U_2\mathbf{C},$
 \begin{gather*}
  Q_1 = \begin{pmatrix}
         U_1 & 0 \\
           0 & U_2
        \end{pmatrix},\quad
  Q_2 = \begin{pmatrix}
         V_1 &     0 &     0 &     0 &   0 \\
           0 & U_2^* &     0 &     0 &   0 \\
           0 &     0 & U_1^* &     0 &   0 \\
           0 &     0 &     0 & U_2^* &   0  \\
           0 &     0 &     0 &     0 & V_2
        \end{pmatrix}
        Q_3, \\
  Q_3 = \begin{pmatrix}
         I_n & 0 & 0 & 0 & 0 \\
           0 & \begin{bmatrix}1\\0\\ \vdots \\0\end{bmatrix} & 0 & 0 & 0 \\
           0 & 0 & I_n & 0 & 0 \\
           0 & 0 & 0 & \begin{bmatrix}0\\ I_n \end{bmatrix} & 0  \\
          0 & 0 & 0 & 0 & I_{n+1}
        \end{pmatrix}Q_4, \\
  Q_4 = \begin{pmatrix}
         I_n &      0 & (-1)^{n+1}C_n^* & 0 & 0 & 0 \\
           0 & \sqrt2 &             0 & 0 & 0 & 0 \\
         I_n &      0 &   (-1)^{n}C_n^* & 0 & 0 & 0 \\
         0 & 0 & 0 & I_n &      0 &  (-1)^{n+1}C_n^* \\
         0 & 0 & 0 &   0 & \sqrt2 &              0 \\
         0 & 0 & 0 & I_n &      0 &    (-1)^{n}C_n^*
        \end{pmatrix}.
 \end{gather*}
\end{theorem*}

\begin{theorem*}[Canonical forms for $2(2k+1)+1$-th order differential operators]
 \strut\\
 Let $2k+1=n\in\mathbb{N}$ be odd and let $A$ and $B$ be $(2n+1)\times (2n+1)$ matrices satisfying 
 \begin{equation*}
  \rank(A:B)=2n+1 \quad \textrm{and}\quad AC_{2n+1}A^*=BC_{2n+1}B^*.
 \end{equation*}
 Then, there exist $n\times n$ unitary matrices $U_1$ and $V_1$,
 $(n+1)\times (n+1)$ unitary matrices $U_2$ and $V_2$
 and real diagonal $n\times n$ matrices $C$ and $S$
 such that
 \begin{enumerate}[label=(\alph*)]
  \item
   $C^2+S^2=I_n$,
  \item
   $(A:B)$ has the canonical form
   \begin{equation*}
   (A:B) \sim \dfrac{1}{\sqrt{2}} Q_1
           \begin{pmatrix}
            \mathbf{C} & I_{n+1} & 0 & I_{n+1} & \mathbf{S} \\
            -\mathbf{S}^* & 0 & I_n & 0 & C
           \end{pmatrix}
           Q_2,
   \end{equation*}
   where $(A:B) \sim (A':B')$ if and only if there exists $J$ such that $J^{-1}(A:B) = (A':B')$,
  \item $\rank(A) = 2n+1-\Null(I_n-KK^*) = \rank(B)$,
  \item the boundary conditions are
   \begin{enumerate}[label=(\roman*)]
    \item mixed, if and only if $\rank(I_n-KK^*)<n$,
    \item coupled, if and only if $\rank(I_n-KK^*)=n$,
   \end{enumerate}
 \end{enumerate}
 \medskip
 where
 $\mathbf{S}=\begin{pmatrix} S & 0 \end{pmatrix}^*$,
 $\mathbf{C}=C\oplus 1=\begin{pmatrix}C&0\\0&1\end{pmatrix}$,
 $K=\begin{pmatrix}I_n&0\end{pmatrix}U_1\mathbf{C},$
 \begin{gather*}
  Q_1 = \begin{pmatrix}
         U_1 & 0 \\
           0 & U_2
        \end{pmatrix},\quad
  Q_2 = \begin{pmatrix}
         V_1 &     0 &     0 &     0 &   0 \\
           0 & U_1^* &     0 &     0 &   0 \\
           0 &     0 & U_2^* &     0 &   0 \\
           0 &     0 &     0 & U_1^* &   0  \\
           0 &     0 &     0 &     0 & V_2
        \end{pmatrix}
        Q_3, \\
  Q_3 = \begin{pmatrix}
         I_{n+1} & 0 & 0 & 0 & 0 \\
           0 & \begin{bmatrix} I_n \\ 0 \end{bmatrix} & 0 & 0 & 0 \\
           0 & 0 & I_n & 0 & 0 \\
           0 & 0 & 0 & \begin{bmatrix}0\\\vdots \\0\\1\end{bmatrix} & 0  \\
           0 & 0 & 0 & 0 & I_n
        \end{pmatrix}Q_4, \\
  Q_4 = \begin{pmatrix}
         I_n &      0 & (-1)^{n+1}C_n^* & 0 & 0 & 0 \\
           0 & \sqrt2 &             0 & 0 & 0 & 0 \\
         I_n &      0 &   (-1)^{n}C_n^* & 0 & 0 & 0 \\
         0 & 0 & 0 & I_n &      0 &  (-1)^{n+1}C_n^* \\
         0 & 0 & 0 &   0 & \sqrt2 &              0 \\
         0 & 0 & 0 & I_n &      0 &    (-1)^{n}C_n^*
        \end{pmatrix}.
 \end{gather*}
\end{theorem*}

%\section{Declarations}
%
%\subsection{Competing interests}
%
%The authors have no competing interests.
%
%\subsection{Data availability}
%
%No datasets were generated or analysed during the current study.

\bibliographystyle{amsplain}
\bibliography{canonicalforms}

\end{document}